\documentclass{emsprocart}
\usepackage{stmaryrd}
\usepackage{xr}
\usepackage{amscd, amssymb, calc, ifthen, mathrsfs}
\usepackage[matrix,arrow,curve,cmtip]{xy}
\usepackage{amsxtra}
\usepackage{eucal}
\usepackage{color}
\usepackage[normalem]{ulem}
\usepackage{amsxtra}

\contact[james.borger@anu.edu.au]{Mathematical Sciences Institute\\ Australian National University\\Canberra ACT 0200\\Australia}

\contact[darijgrinberg@gmail.com]{Massachusetts Institute of Technology\\ 77 Massachusetts Avenue\\ Cambridge, MA 02139-4307\\United States of America}

\let \: \colon
\newcommand{\mathbold}{\mathbb}
\newcommand{\bF}{{\mathbold F}}

\newcommand{\bZ}{{\mathbold Z}}
\newcommand{\bC}{{\mathbold C}}
\newcommand{\bR}{{\mathbold R}}
\newcommand{\bRpl}{{\bR_+}}

\newcommand{\bN}{{\mathbold N}}

\newcommand{\tn}{\otimes}           

\newcommand{\setof}[2]{\{#1 \;|\; #2\}}
\newcommand{\bigsetof}[2]{\big\{#1 \;|\; #2\big\}}
\newcommand{\longmap}{{\,\longrightarrow\,}}

\def\longisomap{{\,\buildrel \sim\over\longrightarrow\,}} 

\newcommand{\eps}{\varepsilon}

\newcommand{\Mod}[1]{\mathsf{Mod}_{{#1}}}
\newcommand{\Alg}[1]{\mathsf{Alg}_{{#1}}}
\newcommand{\Hom}{\mathrm{Hom}}

\newcommand{\Sch}{\Lambda^{\mathrm{Sch}}}
\newcommand{\Wsch}{W^{\mathrm{Sch}}}
\newcommand{\anti}[1]{{\{#1\}}}

\newcommand{\srp}{\sigma}

\newcommand{\Nhat}[1]{\bN[{#1}]{\sphat}\,}

\newcommand{\id}{{\mathrm{id}}}

\newcommand{\Bool}{\mathbold{B}}
\DeclareMathOperator{\pker}{pker}
\DeclareMathOperator{\sker}{pker}
\newcommand{\pleq}{\subseteq}

\newcommand{\pnleq}{\not\subseteq}
\newcommand{\wleq}{\preccurlyeq}
\newcommand{\wgeq}{\succcurlyeq}
\newcommand{\Wp}{W_{(p)}}
\newcommand{\phiSch}{\varphi^{\mathrm{Sch}}}

\newcommand{\rightlabelxyarrows}[2]{{\ar@<0.7ex>^-{#1}[r]\ar@<-0.7ex>_-{#2}[r]}}

\newcommand{\displaylabelfork}[6]{{	\entrymodifiers={+!!<0pt,\fontdimen22\textfont2>}
	\def\objectstyle{\displaystyle}
\xymatrix{{#1} \ar^-{#2}[r] & {#3} \ar@<0.7ex>^-{#4}[r]\ar@<-0.7ex>_-{#5}[r] & {#6}}}}

\newcommand{\predisplaylabelfork}[6]{{{#1} \ar^-{#2}[r] & {#3} \ar@<0.7ex>^-{#4}[r]\ar@<-0.7ex>_-{#5}[r] & {#6}}}

\newtheoremstyle{mythm}{}{}%
  {\itshape}
  {}
  {\bfseries}
  {}
  { }
  {\thmnumber{#2.\hspace{1.5mm}}\thmname{#1}\thmnote{ {\mdseries(#3)}}.}

\newtheorem*{theorem*}{Theorem} 

\newtheoremstyle{intro}{}{}%
  {\itshape}
  {}
  {\bfseries}
  {}
  { }
  {\thmname{#1}\thmnumber{ #2}\thmnote{ #3}.}

\numberwithin{equation}{subsection}
\theoremstyle{mythm}
\newtheorem{theorem}[subsection]{Theorem}
\newtheorem{proposition}[subsection]{Proposition}
\newtheorem{lemma}[subsection]{Lemma}
\newtheorem{corollary}[subsection]{Corollary}

\theoremstyle{intro}

\renewcommand{\theequation}
  {\arabic{section}.\arabic{subsection}.\arabic{equation}}

\title[Boolean Witt vectors and an integral Edrei--Thoma theorem]{Boolean Witt vectors and an integral Edrei--Thoma theorem}
\author[J.~Borger, D.~Grinberg]{James Borger\thanks{Supported the Australian Research Council, DP120103541 and FT110100728.}, Darij Grinberg}

\begin{document}

\begin{abstract}
A subtraction-free definition of the big Witt vector construction was recently given by the first author. This
allows one to define the big Witt vectors of any semiring. Here we give an explicit combinatorial description
of the big Witt vectors of the Boolean semiring. We do the same for two variants of the big Witt vector
construction: the Schur Witt vectors and the $p$-typical Witt vectors. We use this to give an explicit
description of the Schur Witt vectors of the natural numbers, which can be viewed as the classification of
totally positive power series with integral coefficients, first obtained by Davydov. We also determine the
cardinalities of some Witt vector algebras with entries in various concrete arithmetic semirings.

\end{abstract}



\maketitle

\section{Introduction}

Recently in~\cite{Borger:totalpos}, the first author gave an extension of the big Witt vector functor to a
functor $W$ on (commutative) semirings---in other words, a definition of Witt vectors that does not require
subtraction. Witt vectors are usually thought of as being related to the $p$-adic numbers or, more broadly, the
finite adeles. For example, the (relative) de Rham--Witt complex, in its single-prime form, computes crystalline
cohomology; and in its multiple-prime form, it formally unifies the crystalline cohomologies at all primes.
(See, e.g.,~\cite{Bloch:dRW}\cite{Illusie:dRW-1016}\cite{Langer-Zink:dRW}\cite{Chatzistamatiou:dRW}\cite{Hesselholt-Madsen:nilpotent}\cite{Hesselholt:big-dRW}.) But the extension of the Witt construction to semirings shows
that positivity also plays a role, and thus so does the archimedean prime. One striking expression of this is
that $W(\bN)$ is countable and has no zero divisors, and therefore can connect to analytic mathematics in ways
that $W(\bZ)$, which is pro-discrete and has many zero divisors, cannot. For instance, there is a canonical
injective homomorphism from $W(\bN)$ to the ring of entire functions on $\bC$. 
In~\cite{Borger:LRFOE}, it was argued that it is reasonable to think of $W(\bZ)$ as $\bZ\tn_{\bF_1}\bZ$, where
$\bF_1$ is the hypothetical field with one element; so it is notable that $W(\bN)$, a more fundamental
alternative, has some analytic nature.

For reasons like these, one would like to understand better the range of positivity phenomena that can occur in
Witt vectors. Here a special place is held by the Boolean semiring $\Bool=\{0,1\}$, with $1+1=1$. This is
because, aside from the fields, it is the unique nonzero semiring which has no quotients other than itself and
$0$. (See Golan~\cite{Golan:book1}, p.\ 87.) One might imagine $\Bool$ as the residue field of $\bN$ at the
infinite prime. In fact, as we show in~(\ref{prop:nonrings-map-to-Bool}), any semiring $A$ not
containing $-1$ necessarily admits a map to $\Bool$, much as any ring not containing $1/p$
admits a map to some field of characteristic $p$. In particular, there is a map
$W(A)\to W(\Bool)$, and hence structure related to Boolean Witt vectors is necessarily present in the Witt
vectors of any semiring that is not a ring. So not only is an understanding of $W(\Bool)$ helpful in
understanding positivity in Witt vectors, it is all but required.

The purpose of this paper is then to give an explicit description of $W(\Bool)$. From the point of view above,
this would be analogous to the description of Witt vectors of fields of characteristic $p$ in concrete $p$-adic
terms (which incidentally explains why $p$-adic phenomena appear in the Witt vectors of any ring not containing
$1/p$). We also give an explicit description $\Wsch(\Bool)$, where $\Wsch$ is the Schur Witt vector
functor of~\cite{Borger:totalpos}.

In order to state our main results precisely, we need to recall the basic definitions of $W$ and $\Wsch$ for semirings. For details, see~\cite{Borger:totalpos}. Let $\Lambda_\bZ$ denote the usual ring of symmetric functions in the formal variables $x_1,x_2,\dots$. It is the inverse limit in the category of graded rings of the rings $\bZ[x_1,\dots,x_n]^{S_n}$ of symmetric polynomials, where the transition maps are given by $f(x_1,\dots,x_n)\mapsto f(x_1,\dots,x_{n-1},0)$, and where each variable $x_i$ has degree $1$. Then $\Lambda_\bZ$ has two standard $\bZ$-linear bases: the monomial symmetric functions $(m_\lambda)_{\lambda\in P}$ and the Schur symmetric functions $(s_\lambda)_{\lambda\in P}$, both of which are indexed by the set $P$ of partitions. (We refer the reader to chapter 1 of Macdonald's book~\cite{Macdonald:SF} for the basics on symmetric functions.) If we write $\Lambda_\bN=\bigoplus_{\lambda\in P} \bN m_\lambda$ and $\Sch=\bigoplus_{\lambda\in P}\bN s_\lambda$, where $\bN=\{0,1,\dots\}$, then by some standard positivity facts in the theory of symmetric functions, both are subsemirings of $\Lambda_\bZ$ and we have $\Sch\subseteq\Lambda_\bN$. For any semiring $A$, we then define
	$$
	W(A) = \Hom_{\mathrm{semiring}}(\Lambda_\bN,A), \quad \quad \Wsch(A) = \Hom_{\mathrm{semiring}}(\Sch,A).
	$$
We call these the sets of big Witt vectors and Schur Witt vectors with entries in $A$. When $A$ is a ring, both agree with the usual ring of big Witt vectors defined by $\Hom_{\Alg{\bZ}}(\Lambda_\bZ,A)$ in the sense that any map from $\Sch$ or $\Lambda_\bN$ to $A$ extends uniquely to $\Lambda_\bZ$. The semiring structures on $W(A)$ and $\Wsch(A)$ are the unique functorial ones that agree with the usual Witt vector ring structure when $A$ is a ring. Their existence follows from some further standard positivity results in the theory of symmetric functions. 

Our main result is the following combinatorial description of $W(\Bool)$ and $\Wsch(\Bool)$:

\begin{theorem}
\label{thm:intro-B}
\begin{enumerate}
	\item \label{part:W(Bool)-bijection}
		There is a bijection
			\begin{align*}
			\bN^2\cup\{\infty\} &\longisomap W(\Bool)
			\end{align*}
		sending $(x,y)$ to the unique Witt vector satisfying, for all $\lambda\in P$,
				\begin{equation}
				m_\lambda \longmapsto
					\begin{cases}
					0 & \text{if }\lambda_{x+1} \geq y+1 \\ 
					1 & \text{otherwise,} 			
					\end{cases}		
				\end{equation}
		and sending $\infty$ to the unique Witt vector satisfying $m_\lambda\mapsto 1$ for all $\lambda\in P$.
	\item \label{part:W(Bool)-semiring}
		In these terms, the semiring structure on $W(\Bool)$ is given by the laws
		\begin{gather*}
		(x,y)+(x',y') = (x+x',\max\{y,y'\}), \quad (0,y)\cdot (0,y') = (0,\min\{y,y'\}), \\
		0 = (0,0),\quad 1=(1,0),\quad  \infty+z=\infty=\infty\cdot\infty,\quad  \infty\cdot (0,y) = (0,y),
		\end{gather*}
		and the pointwise partial order on $W(\Bool)$ (with $0<1$) corresponds to the usual, componentwise 
		partial order on $\bN^2$ with terminal element $\infty$.
	\item \label{part:Wsch(Bool)-semiring}
		The obvious analogues of parts (1) and (2) for $\Wsch(\Bool)$ and the Schur basis 
		$(s_\lambda)_{\lambda\in P}$ also hold except that the semiring structure on $\Wsch(\Bool)$ is given
		by the laws
		\begin{gather*}
		(x,y)+(x',y') = (x+x',y+y'), \quad (0,y)\cdot (0,y') = (yy',0), \\
		0 = (0,0),\quad 1=(1,0),\quad  \infty+z=\infty=\infty\cdot\infty,\quad  \infty\cdot (0,y) = \infty.
		\end{gather*}
\end{enumerate}
\end{theorem}

The proof is in sections \ref{sect:Wbool}--\ref{sec:Wsch-Bool}. We also describe the $\Lambda_\bN$-semiring structure on $W(\Bool)$ and the $\Sch$-semiring structure on $\Wsch(\Bool)$. Observe that both kinds of Boolean Witt vectors form countably infinite sets. For rings, this never happens. 

At this point, it should be mentioned that Connes has also proposed a notion of Witt vectors with entries in
certain $\Bool$-algebras~\cite{Connes:Witt-characteristic-1}. See also
Connes--Consani~\cite{Connes-Consani:universal-thickening}. His construction is an analogue for $p=1$ of the
$p$-typical Witt vector ring $\Wp(K)$ with entries in a perfect field $K$ of characteristic $p$. Both the
characteristic of the input ring $K$ and the typicality of the Witt construction are imagined to be $1$. In
these terms, our Boolean Witt vectors have only the characteristic of the input set to $1$. The Witt vector
functors themselves are the usual big (and, below, $p$-typical) functors, although extended to semirings. It
would be interesting to know if there are any relations between the two approaches.

In sections 5--6, we use the combinatorial description of Boolean Witt vectors above to determine $\Wsch(\bN)$.
Our approach is by analyzing the diagram
	\begin{equation}
	\label{diag:main-power-series}
	\begin{split}
		\xymatrix{
		1+t\bZ[[t]] \ar[r] & 1+t\bRpl[[t]] \ar[r] & 1+t\Bool[[t]] \\
		\Wsch(\bN) \ar[r]\ar@{>->}^{\srp}[u] & \Wsch(\bRpl)\ar[r]\ar@{>->}^{\srp}[u] & \Wsch(\Bool) \ar^{\srp}[u]
		}
	\end{split}
	\end{equation}
where $\bRpl=\setof{x\in\bR}{x\geq 0}$. The maps $\srp$ send any Witt vector $a$ to the series $\sum_{n\geq 0} a(e_n)t^n$, where 
	$$
	e_n=\sum_{i_1<\dots< i_n} x_{i_1}\cdots x_{i_n} \in \Sch
	$$
is the $n$-th elementary symmetric function, and the horizontal maps are the functorial ones.  (The semiring map $\bRpl\to\Bool$ is the unique one. It sends $0$ to $0$ and all $x>0$ to $1$.) The proof is a short argument linking the combinatorial description of $\Wsch(\Bool)$ given above with an analytical description of the real column in~(\ref{diag:main-power-series}) given by the Edrei--Thoma theorem from the theory of total positivity, as recalled in section~\ref{sec:total-pos}. The result is the following:
\begin{theorem}
\label{thm:intro-A}
The map $\srp$ is a bijection from $\Wsch(\bN)$ to the set of integral series $f(t)\in 1+t\bZ[[t]]$ 
of the form
	$$
	f(t)=\frac{g(t)}{h(t)},
	$$
for some (necessarily unique) integral polynomials $g(t),h(t)\in 1+t\bZ[t]$ such that all the complex roots 
of $g(t)$ are negative real numbers and all those of $h(t)$ are positive real numbers. 
\end{theorem}

This theorem is obviously equivalent to the following one, which was apparently first observed by Davydov~\cite{Davydov:totally-positive-sequences} (end of the proof of theorem 1) and deserves to be more widely known:
\begin{theorem}
\label{thm:Davydov}
An integral series $f(t)\in 1+t\bZ[[t]]$ is totally positive if and only if it is
of the form 
	$$
	f(t)=\frac{g(t)}{h(t)},
	$$
for some integral polynomials $g(t),h(t)\in 1+t\bZ[t]$ such that all the complex roots 
of $g(t)$ are negative real numbers and all those of $h(t)$ are positive real numbers. 
\end{theorem}

Davydov's argument also uses the Edrei--Thoma theorem, but where we use the combinatorial calculation of $\Wsch(\Bool)$, he uses a theorem of Salem's~\cite{Salem:integral-coefficients}.

Let us now return to theorem~(\ref{thm:intro-B}). Observe that $W(\Bool)$ and $\Wsch(\Bool)$ are isomorphic as partially ordered sets, since both agree with $\bN^2\cup\{\infty\}$, but that they are not isomorphic as semirings. The reason they both agree with $\bN^2\cup\{\infty\}$ as sets is ultimately that nonzero rectangular partitions, which are indexed by $\bN^2$, have a certain primality role in the combinatorics of both monomial and Schur symmetric functions. In fact, most of the arguments for $W(\Bool)$ and for $\Wsch(\Bool)$ are remarkably parallel. It would be interesting to have a conceptual reason why this should be so.

This surprising coincidence continues with the $p$-typical Witt vectors,
which are the subject of section~\ref{sec:p-typical}. Let $p$ be a prime number,
and let $\Wp$ denote the $p$-typical Witt functor for semirings, as defined in~\cite{Borger:totalpos}
and recalled in~(\ref{subsec:p-typical}) below.

\begin{theorem}
$\Wp(\Bool)$ is canonically isomorphic to $\bN^2\cup\{\infty\}$ as a partially ordered set. With respect to
this identification, the semiring structure on $\Wp(\Bool)$ is independent of $p$, for $p>2$. 
\end{theorem}

Here the connection with partitions is lost. So perhaps it is unreasonable to expect the coincidence in
this case to have a deeper meaning. The semiring structure on $\Wp(\Bool)$ is described explicitly in~(\ref{prop:p-typical}) below. It is not isomorphic to either $W(\Bool)$ or $\Wsch(\Bool)$.

In section~\ref{sec:countability}, we establish some countability results for Witt vectors of other semirings, such as $\bN[x]$ and the truncations $\bN/(n=n+1)$ of $\bN$. The final surprise of this paper is that for the truncations, there is a jump at $n=3$:

\begin{theorem}
$W(\bN/(n=n+1))$ is countable if and only if $n\leq 2$. The same is true for $\Wsch(\bN/(n=n+1))$.
\end{theorem}

It would be interesting to improve on such cardinality statements and determine the algebraic structure of more Witt vector algebras explicitly, as we have done for Boolean Witt vectors. One could start with $\bN[x]$ or $\bN/(2=3)$, for example.

\section{Background and conventions}
\label{sec:background}

We will follow the conventions and notations of~\cite{Borger:totalpos}. Let us recall some
of the ones that have not already appeared in the introduction.

\subsection{Modules and algebras over $\bN$}
The set $\bN$ of natural numbers is $\{0,1,\dots\}$.
The category $\Mod{\bN}$ of $\bN$-modules is by definition the category of commutative monoids. The
monoid operation is denoted $+$, and the identity is $0$. The set of $\bN$-module homomorphisms
$M\to N$ will be denoted $\Mod{\bN}(M,N)$, and similarly for other categories.

An $\bN$-algebra (or a semiring) is an $\bN$-module $A$ with a second commutative monoid operation
$\times$ which is bilinear with respect to $+$, in the sense that for any $a\in A$, the map
$x\mapsto a\times x$ is an endomorphism of the monoid $(A,+)$. The identity of $\times$ is denoted
$1$, and $a\times x$ is typically written $ax$. The category of $\bN$-algebras is denoted
$\Alg{\bN}$. 

The Boolean semiring $\Bool$ is defined to be the quotient $\bN/(1+1=1)$, or simply $\{0,1\}$ with $1+1=1$.

\subsection{Partitions and symmetric functions}
We will generally follow Macdonald's book~\cite{Macdonald:SF}. In particular,
a partition $\lambda$ is a sequence
$(\lambda_1,\lambda_2,\dots)\in \bN^\infty$ such that $\lambda_1\geq\lambda_2\geq\cdots$
and $\lambda_i=0$ for sufficiently large $i$. The set of partitions will be denoted $P$.
We will write $a^b$ for the partition $(a,a,\dots,a,0,\dots)$ with $a$ occurring $b$ times.
The $n$-th power-sum symmetric function in $\Lambda_\bN$ will be denoted by $\psi_n$:
	$$
	\psi_n = \sum_i x_i^n
	$$
instead of the more common $p_n$.

\subsection{Witt vectors}
As mentioned in the introduction, $W(A)$ is an $\bN$-algebra for any $\bN$-algebra $A$.
The binary operations $+,\times$ on $W(A)$ are induced by co-operations on the representing 
object $\Lambda_\bN$:
	\begin{equation}
		\Delta^+, \Delta^\times\: \Lambda_\bN \longmap \Lambda_\bN\tn_\bN\Lambda_\bN.
	\end{equation}
In terms of symmetric functions in the formal variables $x_1,x_2,\dots$, they
are given by the following rules: 
	\begin{align*}
		\Delta^+: f(\dots, x_i,\dots) &\mapsto f(\dots,x_i\tn 1, 1\tn x_i,\dots) \\
		\Delta^\times: f(\dots,x_i,\dots) &\mapsto f(\dots,x_i\tn x_j,\dots).
	\end{align*}
For example, we have 
	\begin{align*}
		\Delta^+(\psi_n)&=\psi_n\tn 1 + 1\tn \psi_n, \\
		\Delta^\times(\psi_n)&=\psi_n\tn\psi_n.
	\end{align*}
These equations are the properties of $\Delta^+$ and $\Delta^\times$ we will most often use, and in fact
they determine $\Delta^+$ and $\Delta^\times$ uniquely and hence furnish alternative definitions.

Similarly, the additive and multiplicative units $0,1\in W(A)$ are induced by co-units
	\begin{equation}
		\eps^+, \eps^\times\: \Lambda_\bN \longmap \bN
	\end{equation}
given by
	\begin{align*}
		\eps^+: f(\dots, x_i,\dots) &\mapsto f(0,0,0,\dots) \\
		\eps^\times: f(\dots,x_i,\dots) &\mapsto f(1,0,0,\dots).
	\end{align*}
The same statements hold for $\Sch$ instead of $\Lambda_\bN$.
	
For any $x\in A$, the \emph{anti-Teichm\"uller} lift $\anti{x}\in\Wsch(A)$ is defined by
	\begin{equation}
	\label{eq:anti}
	\anti{x}:s_\lambda \mapsto 
		\begin{cases}
			x^r & \text{if $\lambda=1^r$} \\
			0 & \text{otherwise.}
		\end{cases}		
	\end{equation}
So for example, we have 
	\begin{equation}
		\srp(\{x\})=1+xt+x^2t^2+\cdots = (1-xt)^{-1}.
	\end{equation}  
See \cite{Borger:totalpos},~(6.10) for the basic properties of 
anti-Teichm\"uller lifts. (Note that
the symbol $\srp$ denotes a different map in~\cite{Borger:totalpos}, corresponding to different sign 
conventions. There two maps $\srp^+_+$ and $\srp^-_-$ 
are defined. Here $\srp$ means $\srp^+_+$, whereas there it means $\srp^-_-$.)

We end with some words on composition structures, which will have a small part in this paper.
Consider a composition $\bN$-algebra $Q$, such as $\Lambda_\bN$ or $\Sch$.
(See (3.3) in~\cite{Borger:totalpos} for the general definition.)  
Then the $Q$-Witt vectors
$W_Q(A)=\Alg{\bN}(Q,A)$ have not only an algebra structure but also an action of $Q$. For $f\in Q$, 
$a\in W_Q(A)$, the element $f(a)\in W_Q(A)$ is the map $Q\to A$ given by
	$$
	f(a)\: g\mapsto a(g\circ f),
	$$
where $\circ$ denotes the composition operation on $Q$---this 
is simply plethysm of symmetric functions in all examples in this paper. In fact, $W_Q$ is the right adjoint of 
the forgetful functor from $Q$-semirings to
semirings. This generalizes the fact that the big Witt functor is the right adjoint of the forgetful
functor from $\lambda$-rings to rings.

\section{Orders on Witt vectors}

This section introduces some generalities on pre-orders. They
have only a small part in this paper; so it can be skipped and consulted as needed.

\subsection{$\Alg{\bN}$-preorders}
Recall that a preorder is a reflexive transitive relation.
Let us say that a relation $\leq$ on an $\bN$-algebra $A$ is an $\Alg{\bN}$-preorder if
it is a preorder and satisfies the implication
	\begin{equation}
		a\leq b\text{ and } a'\leq b' \Longrightarrow a+a'\leq b+b'\text{ and } aa'\leq bb'.
	\end{equation}
This implication is equivalent to requiring that the graph of $\leq$ be a
sub-$\bN$-algebra of $A\times A$. Hence an $\Alg{\bN}$-preorder is the same as a preorder object
over $A$ in the category $\Alg{\bN}$.

The most important example for this paper is the \emph{difference preorder} defined by 
	\begin{equation}
		\label{eq:difference-preorder-definition}
		a \leq b \Longleftrightarrow \exists c\in A \text{ such that }b=a+c.
	\end{equation}
It is the strongest $\Alg{\bN}$-preorder such that $c\geq 0$ for all $c\in A$.
A second example is $a\leq b \Leftrightarrow a=b$; and a third is $a\leq b$ for all $a,b$,
which agrees with (\ref{eq:difference-preorder-definition}) when $A$ is a ring.

\subsection{Preorders on Witt vectors}
Let $Q$ be a composition $\bN$-algebra, for example $\Lambda_\bN$ or $\Sch$.
(See (3.3) in~\cite{Borger:totalpos} for the general definition.) 
Then given any preorder $\leq$ on an $\bN$-algebra $A$,
we can define a preorder on the algebra $W_Q(A)=\Hom_{\Alg{\bN}}(Q,A)$ of Witt 
vectors: for $a,b\in W_Q(A)$, we write 
	\begin{equation}
		\label{property:order-on-W}
		a\wleq b \Longleftrightarrow a(x)\leq b(x) \text{ for all } x\in Q.
	\end{equation}
Observe that if $\leq$
is a partial order (that is, it satisfies $a\leq b\leq a \Rightarrow a=b$), then so is $\wleq$.

Now assume that $\leq$ is an $\Alg{\bN}$-preorder. Then $\wleq$ is also an $\Alg{\bN}$-preorder. Indeed,
if $R\subseteq A\times A$ denotes the graph of $\leq$, then $W_Q(R)$ is a subalgebra of $W_Q(A\times A)
=W_Q(A)\times W_Q(A)$ and is clearly the graph of $\wleq$. In fact, $W_Q(R)$ is a sub-$Q$-semiring of
$W_Q(A)\times W_Q(A)$. So $\wleq$ is a preorder in the category of $Q$-semirings. Equivalently, it is an
$\Alg{\bN}$-preorder satisfying
	\begin{equation}
		\label{eq:P-preorder}
		a\wleq b \Longrightarrow f(a)\wleq f(b)
	\end{equation}
for any $f\in Q$.

Finally observe that if a subset $S\subseteq Q$ generates $Q$ as an $\bN$-algebra, then for all 
$a,b\in W_Q(A)$, we have
	\begin{equation}
		\label{eq:order-determined-by-generators}
		a \wleq b \Longleftrightarrow a(s) \leq b(s) \text{ for all } s\in S.
	\end{equation}
Indeed, since $R$ is a subalgebra of $A\times A$, the map $(a,b)\:Q\to A\times A$ factors through
$R$ if and only if the image of $S$ is contained in $R$. 

\subsection{The preorders in this paper}
The preorder $\leq$ on $\bN$-algebras will always be understood to be the difference
preorder of (\ref{eq:difference-preorder-definition}), and the preorder $\wleq$ on Witt vectors
$W_Q(A)$ will always be understood to be the one induced by the difference preorder on $A$.

The following result implies that for all the composition $\bN$-algebras $Q$ that appear in this
paper---namely $\Lambda_\bN$, $\Sch$, and $\Lambda_{\bN,(p)}$---the difference preorder $\leq$ on $W_Q(A)$
is stronger than $\wleq$.

\begin{proposition}\label{pro:inequality-comparison}
For $a,c\in W_Q(A)$, we have 
	\begin{equation}
	\label{eq:inequality-comparison}
	a\wleq a+c,
	\end{equation}
as long as $Q$ has the property that for all $x\in Q$, we have $\eps^+(x)\leq x$,
where $\eps^+\:Q\to\bN$ denotes the additive co-unit of $Q$. 
\end{proposition}
\begin{proof}
Given any element
$z= \sum_i x_i\tn y_i\in Q\tn_\bN Q$, write $x_i=\eps^+(x_i)+x'_i$, for some elements $x'_i\in Q$.
Then we have
$$z = \sum_i (\eps^+(x_i)+x'_i)\tn y_i = \sum_i\eps^+(x_i)\tn y_i + \sum_i x'_i\tn y_i
\geq (\eps^+\tn\id)(z).$$
When $z=\Delta^+(x)$
for some $x\in Q$, this simplifies to $\Delta^+(x)\geq 1\tn x$. 
Applying the homomorphism $c\tn a\: Q\tn_\bN Q \to A$ to this inequality, we obtain
$(c+a)(x)\geq a(x)$. Since this holds for all $x\in Q$, we have $c+a\wgeq a$.
\end{proof}

\section{$W(\Bool)$}
\label{sect:Wbool}

The purpose of this section is to determine $W(\Bool)$.

\subsection{Partition ideals}
Recall that $P$ denotes the set of partitions. Let us say that a subset $I\subseteq P$ is a 
\emph{partition ideal} if the following implication holds for all $\lambda,\mu\in P$:
	$$
	\lambda\in I, \lambda\pleq\mu \Longrightarrow \mu\in I.
	$$
Here, $\pleq$ stands for the inclusion ordering of partitions, i.e., two partitions $\lambda$ and
$\mu$ satisfy $\lambda\pleq\mu$ if and only if $\lambda_i\leq \mu_i$ for all $i$ (including those
for which $\mu_i=0$).

For any symmetric function $g\in\Lambda_\bN$, let us say the \emph{(monomial) constituents} of $g$
are the partitions $\nu$ such that $g_\nu\neq 0$ in the unique decomposition $g=\sum_\nu g_\nu
m_\nu$, $g_\nu\in\bN$, where the $m_\nu$ are the monomial symmetric functions. The reason for the
optional modifier \emph{monomial} here and below is that we will consider Schur analogues of these
concepts in the next section.

Let us say that a partition ideal $I$ is \emph{(monomial) prime} if $I\neq P$ and whenever
partitions $\lambda,\mu\in P$ have the property that every constituent of $m_\lambda m_\mu$ is in
$I$, then either $\lambda$ or $\mu$ lies in $I$.

For any linear map $f\in\Mod{\bN}(\Lambda_\bN,\Bool)$, write 
	$$
	\pker(f)=\setof{\lambda\in P}{f(m_\lambda)=0}.
	$$
We will call $\pker(f)$ the \emph{(monomial) partition kernel} of $f$.

Understanding $W(\Bool)$ hinges on knowing the constituents of the product of two monomial
symmetric functions. This is a combinatorial problem which is not very deep:

\begin{proposition}
\label{pro:W(Bool)-comblem1}
Let $\lambda$ and $\mu$ be two partitions. Then the constituents of the product $m_{\lambda}
m_{\mu}$ are the partitions $\nu$ which can be obtained by picking a permutation
$\sigma$ of the set $\left\lbrace 1,2,3,\dots\right\rbrace$ and rearranging the sequence
$$\left(\lambda_1 + \mu_{\sigma\left(1\right)}, \lambda_2 + \mu_{\sigma\left(2\right)}, \lambda_3 +
\mu_{\sigma\left(3\right)}, \dots\right)$$ in weakly decreasing order.
\end{proposition}

\begin{proof}
Clear.
\end{proof}

\begin{corollary}\label{cor:W(Bool)-comblem2}
Let $\lambda$ and $\mu$ be two partitions. Then, every constituent $\nu$ of $m_{\lambda} m_{\mu}$ satisfies $\lambda\pleq \nu$ and $\mu\pleq \nu$.
\end{corollary}
\begin{proof}
Let $\nu$ be a constituent of $m_{\lambda} m_{\mu}$. Due to (\ref{pro:W(Bool)-comblem1}), there
exists a permutation $\sigma$ of the set $\left\lbrace 1,2,3,...\right\rbrace$ such that $\nu$ is
the weakly decreasing permutation of the sequence $\left(\lambda_1 + \mu_{\sigma\left(1\right)},
\lambda_2 + \mu_{\sigma\left(2\right)}, \lambda_3 + \mu_{\sigma\left(3\right)}, ...\right)$. But
for each positive integer $i$, at least $i$ terms of the sequence $\left(\lambda_1 +
\mu_{\sigma\left(1\right)}, \lambda_2 + \mu_{\sigma\left(2\right)}, \lambda_3 +
\mu_{\sigma\left(3\right)}, ...\right)$ are $\geq \lambda_i$ (since at least $i$ terms of the
sequence $\left(\lambda_1,\lambda_2,\lambda_3,...\right)=\lambda$ are $\geq \lambda_i$). Hence, for
each positive integer $i$, at least $i$ terms of $\nu$ (which is a permutation of this sequence)
are $\geq \lambda_i$. Since $\nu$ is a partition, this entails $\nu_i\geq\lambda_i$. Thus,
$\lambda\pleq\nu$, and similarly $\mu\pleq\nu$.
\end{proof}

The next proposition is more technical and will be used to classify the prime partition ideals:

\begin{proposition}
\label{pro:W(Bool)-comblem3}
Let $\lambda$ and $\mu$ be two partitions, and $i$ and $j$ be two positive integers such that
$\lambda_i > \mu_i$, $\lambda_i > \lambda_{i+1}$, $\mu_j > \lambda_j$, and $\mu_j > \mu_{j+1}$. Let
$\bar{\lambda}$ denote the result of diminishing the $i$-th part of $\lambda$ by $1$, and let $\bar{\mu}$ be
the result of diminishing the $j$-th part of $\mu$ by $1$. Then, for every constituent $\nu$ of
$m_{\bar{\lambda}} m_{\bar{\mu}}$, we have either $\lambda\pleq\nu$ or $\mu\pleq\nu$.
\end{proposition}

\begin{proof}
Since $\lambda_i > \mu_i$ and $\mu_j > \lambda_j$, we have $i\neq j$. Thus, we can assume by 
symmetry that $i<j$.

Fix a constituent $\nu$ of $m_{\bar{\lambda}} m_{\bar{\mu}}$. Putting (\ref{pro:W(Bool)-comblem1})
into words, we see that the constituents of $m_{\bar{\lambda}} m_{\bar{\mu}}$ are found by
combining rows of $\bar{\lambda}$ with rows (possibly with zero boxes) of $\bar{\mu}$, such that no
two rows of $\bar{\lambda}$ are combined with the same row of $\bar{\mu}$ and vice versa. Let us
consider the matching between rows of $\bar{\lambda}$ and rows of $\bar{\mu}$ that gives rise to
the constituent $\nu$. We distinguish between two cases, depending on whether (i) the $i$-th row of
$\bar{\lambda}$ is combined with a nonzero-length row of $\bar{\mu}$ or (ii) with a zero-length row
of $\bar{\mu}$.

\emph{Case (i).} If the $i$-th row of $\bar{\lambda}$ is combined with a row of $\bar{\mu}$ with at
least one box, then the resulting row has at least $\bar{\lambda}_i+1=\lambda_i$ boxes. Hence, in
this case, the partition $\nu$ has at least $i$ rows of length $\geq \lambda_i$ (namely, the row
obtained by combining the $i$-th row of $\bar{\lambda}$ with its match, and also the rows obtained
by combining the first $i-1$ rows of $\bar{\lambda}$ with their matches). In other words,
$\nu_i\geq \lambda_i$. Together with $\bar{\lambda}\pleq\nu$ (which follows from
(\ref{cor:W(Bool)-comblem2})), this yields $\lambda\pleq\nu$, which completes the proof in case
(i).

\emph{Case (ii).} If the $i$-th row of $\bar{\lambda}$ is combined with a row of $\bar{\mu}$ with
zero boxes, then the resulting row has $\bar{\lambda}_i\geq\mu_i\geq \mu_j$ boxes. But the first
$j-1$ rows of $\bar{\mu}$, whatever they are combined with, also have at least
$\bar{\mu}_{j-1}=\mu_{j-1}\geq\mu_j$ boxes. So there are at least $j$ rows in $\nu$ with at least
$\mu_j$ boxes. Consequently, $\nu_j\geq \mu_j$. Together with $\bar{\mu}\pleq\nu$ (this is a
consequence of (\ref{cor:W(Bool)-comblem2})), this results in $\mu\pleq\nu$, which completes the
proof in case (ii). 
\end{proof}

Next we include a fact we will not use until section~\ref{sec:countability}:

\begin{proposition}\label{pro:W(Bool)-lr-extremals}
Let $\lambda$ and $\mu$ be two partitions.
	\begin{enumerate}
		\item Let $\lambda + \mu$ denote the partition 
			$\left(\lambda_1+\mu_1,\lambda_2+\mu_2,\lambda_3+\mu_3,...\right)$. Then, 
			$\lambda + \mu$ is a constituent of $m_\lambda m_\mu$.
		\item \label{part:uplus}
			Let $\lambda \uplus \mu$ denote the partition obtained by sorting the components of 
			the vector $\left(\lambda_1,\mu_1,\lambda_2,\mu_2,\lambda_3,\mu_3,...\right)$ in weakly 
			decreasing order. Then, $\lambda \uplus \mu$ is a constituent of $m_\lambda m_\mu$.
		\item If $\lambda$ and $\mu$ are both nonzero, then $m_\lambda m_\mu$ has at least two 
			distinct nonzero constituents.
	\end{enumerate}
\end{proposition}

\begin{proof}
(1): Apply (\ref{pro:W(Bool)-comblem1}) where $\sigma$ is the identity permutation.

(2): Apply (\ref{pro:W(Bool)-comblem1}) where $\sigma$ has the property that $\sigma(1),\dots,\sigma(b)$ are all greater than the length of $\lambda$, where $b$ is the length
of $\mu$.

(3): Since $\lambda$ and $\mu$ are both nonzero, the partitions $\lambda + \mu$ and $\lambda \uplus
\mu$ have different lengths, and are thus distinct. (The length of the former is the maximum of the
lengths of $\lambda$ and $\mu$, while the length of the latter is the sum.) They are also nonzero
and, by parts (1) and (2), constituents of $m_\lambda m_\mu$. 
\end{proof}

\begin{proposition}
\label{pro:W(Bool)-primeparid}
Let $2^P$ denote the set of subsets of $P$.
	\begin{enumerate}
		\item The map $f\mapsto \pker(f)$ is a bijection $\Mod{\bN}(\Lambda_\bN,\Bool)\to 2^P$.
		\item $\pker(f)$ is a partition ideal if and only if $f$ satisfies
			$f(x)=0 \Rightarrow f(xy)=0$.
		\item $\pker(f)$ is a prime partition ideal if and only if $f$ is an $\bN$-algebra map.
	\end{enumerate}
\end{proposition}
\begin{proof}
(1): Linear maps $\Lambda_\bN\to\Bool$ are in bijection with set maps $P\to\Bool$, which are
in bijection with subsets of $P$.

(2): From (\ref{pro:W(Bool)-comblem1}), it follows that for any partition 
$\lambda$, we have
	$$
	m_1 m_\lambda = \sum_\mu a_\mu m_\mu,
	$$
where the sum is over all partitions $\mu$ obtained by adding a single box to (the Young diagram 
of) $\lambda$, and the $a_\mu$ are at least $1$.
Therefore we have 
	\begin{equation*}
	\label{eq:local-2385}
	f(m_1m_\lambda) = \sum_\mu a_\mu f(m_\mu).
	\end{equation*}
If $\lambda\in\pker(f)$ and if $f$ satisfies $f(x)=0 \Rightarrow f(xy)=0$, then we have $f(m_1m_\lambda)=0$ and hence
$f(m_\mu)=0$ for all $\mu$ obtained by adding one box to $\lambda$. It then follows by induction
that $f(m_\mu)=0$ for all $\mu$ with $\lambda\pleq \mu$, which means that $\pker(f)$ is a partition
ideal.

Conversely, suppose $\pker(f)$ is a partition ideal. Let $\lambda\in\pker(f)$. Since for any
partition $\mu$, every constituent $\nu$ of $m_\lambda m_\mu$ satisfies $\lambda\pleq\nu$ (by
 (\ref{cor:W(Bool)-comblem2})), we have $f(m_\lambda m_\mu)=0$. The general implication
$f(x)=0 \Rightarrow f(xy)=0$ then follows. Indeed, write $x=\sum_\lambda x_\lambda m_\lambda$,
$y=\sum_\mu y_\mu m_\mu$. Since $f(x)=0$, it follows that whenever $x_\lambda\neq 0$, we have
$\lambda\in\pker(f)$ and hence $f(m_\lambda m_\mu)=0$. Therefore we have
	$$
	f(xy)=\sum_{\lambda,\mu} x_\lambda y_\mu f(m_\lambda m_\mu) = 0.
	$$

(3): First observe that under either assumption, we have $\pker(f)\neq P$ and $f(1)=1$. Indeed,
if $f$ is an $\bN$-algebra map, then we have $1=f(1)=f(m_0)$ and hence $\pker(f)\neq P$. 
Conversely, if $\pker(f)$ is a prime partition ideal, then there is a partition $\lambda$ such that 
$\lambda\not\in\pker(f)$. Since $\pker(f)$ is an ideal, we have $0\not\in\pker(f)$ and hence that 
$f(1)=f(m_0)=1$.

Now suppose that $f$ is an $\bN$-algebra map and that $\lambda,\mu$ are partitions such that every
constituent of $m_\lambda m_\mu$ is in $\pker(f)$. Then we have $f(m_\lambda)f(m_\mu)=f(m_\lambda
m_\mu)=0$. It follows that either $\lambda\in\pker(f)$ or $\mu\in\pker(f)$, and hence that
$\pker(f)$ is prime.

Conversely, suppose $\pker(f)$ is prime. By part (2), it is enough to show that if $f(x)=f(y)=1$
then $f(xy)=1$. In this case, $x$ has a constituent $\lambda\not\in\pker(f)$ and $y$ has a
constituent $\mu\not\in\pker(f)$. Since $\pker(f)$ is prime, $m_\lambda m_\mu$ must have a
constituent not in $\pker(f)$. This is also a constituent of $xy$. Therefore $f(xy)=1$.
\end{proof}

\begin{proposition}
\label{pro:W(Bool)-primeids}
For integers $x,y\in\bN$, consider the partition ideal
	$$
	I_{(x,y)}=\setof{\lambda\in P}{(y+1)^{x+1}\pleq \lambda},
	$$
where $(y+1)^{x+1}$ denotes the partition $(y+1,\dots,y+1)$ with $y+1$ occurring $x+1$ times.
Then $I_{(x,y)}$ is prime. Conversely, every nonempty prime partition ideal is of this form.	
\end{proposition}
\begin{proof}
Let us first show that $I_{(x,y)}$ is prime. Let $\lambda$ and $\mu$ be partitions, and suppose all the constituents of $m_\lambda m_\mu$ are in $I_{(x,y)}$. By symmetry, we may assume $\lambda_{x+1}\geq\mu_{x+1}$. We will show $\lambda\in I_{(x,y)}$. Let $\nu$ denote the partition obtained by sorting the
components of the vector
	$$
	(\lambda_1+\mu_1,\dots,\lambda_{x}+\mu_x,\lambda_{x+1},\mu_{x+1},\lambda_{x+2},\mu_{x+2},\dots)
	$$
in weakly decreasing order. Then $\nu$ is a constituent of $m_\lambda m_\mu$ (due to  
(\ref{pro:W(Bool)-comblem1})), and hence we have $\nu\in I_{(x,y)}$. Further
the sorting does not change the first $x+1$ components, and so we have 
$\lambda_{x+1}=\nu_{x+1}\geq y+1$ and hence $\lambda\in I_{(x,y)}$.

For the converse, let $I$ be a nonempty prime partition ideal. Let us first show that $I$ contains
a unique minimal element with respect to $\pleq$. Since $I$ is nonempty, there exists some minimal
element. Now suppose for a contradiction that we have two distinct minimal elements $\lambda,\mu\in
I$. Then $\lambda\pnleq\mu$, and so there exists an integer $i$ such that $\lambda_i > \mu_i$ and
$\lambda_i > \lambda_{i+1}$. Likewise, there is a $j$ such that $\mu_j > \lambda_j$ and $\mu_j >
\mu_{j+1}$. Now let $\bar{\lambda}$ denote the result of diminishing the $i$-th part of $\lambda$ by
$1$, and let $\bar{\mu}$ be the result of diminishing the $j$-th part of $\mu$ by $1$. By
(\ref{pro:W(Bool)-comblem3}), for every constituent $\nu$ of $m_{\bar{\lambda}} m_{\bar{\mu}}$, we
have either $\lambda\pleq\nu$ or $\mu\pleq\nu$, and hence $\nu\in I$. Therefore, since the
partition ideal $I$ is prime, either $\bar{\lambda}$ or $\bar{\mu}$ must lie in $I$, and so either $\lambda$
or $\mu$ is not minimal.

It remains to show that the minimal element $\pi$ of $I$ is of the form $(y+1)^{x+1}$. This is
equivalent to showing that if $\pi$ is contained in the union of two partitions $\lambda$ and $\mu$
(that is, we have $\pi_i\leq\lambda_i$ or $\pi_i\leq\mu_i$ for all $i$), then we have
$\pi\pleq\lambda$ or $\pi\pleq\mu$. But for any constituent $\nu$ of $m_\lambda m_\mu$, we have
$\lambda\pleq\nu$ and $\mu\pleq\nu$, by (\ref{cor:W(Bool)-comblem2}). It follows that
$\pi_i\leq\nu_i$ for all $i$. In other words, $\nu$ lies in $I$. Since $I$ is prime, we have
$\lambda\in I$ or $\mu\in I$, and hence $\pi\pleq\lambda$ or $\pi\pleq\mu$. \end{proof}

\subsection{Explicit description of $W(\Bool)$}
\label{subsec:explicit-W(Bool)}
Here we prove parts~(\ref{part:W(Bool)-bijection})--(\ref{part:W(Bool)-semiring}) of theorem~(\ref{thm:intro-B}). Following the notation there, for $x,y\in\bN$, let $(x,y)$ denote the element of $W(\Bool)=\Alg{\bN}(\Lambda_\bN,\Bool)$ corresponding under (\ref{pro:W(Bool)-primeparid}) to the prime partition ideal $I_{(x,y)}$. Thus we have
	$$
	(x,y)\: m_\lambda\mapsto 
	\left\{ 
		\begin{array}{ll} 
			0 & \text{if }\lambda_{x+1} \geq y+1 \\ 1 & \text{otherwise} 
		\end{array} 
		\right.
	$$
Also let $\infty\in W(\Bool)$ denote the element corresponding to the empty partition 
ideal, which is vacuously prime; so $\infty\:m_\lambda\mapsto 1$.	
This defines a map 
	\begin{equation}
	\label{map:W(Bool)-coords}
	\bN^2\cup\{\infty\} \longisomap W(\Bool),
	\end{equation}
which is a bijection by~(\ref{pro:W(Bool)-primeparid})--(\ref{pro:W(Bool)-primeids}).
In these terms, the $\Lambda_\bN$-semiring structure on $W(\Bool)$ is given by the
following proposition.

\begin{proposition}
\label{pro:W(Bool)-description}
	\begin{enumerate}
		\item \label{part:N-to-W(Bool)}
			Under the unique $\bN$-algebra map $\bN\to W(\Bool)$, the image of an element $x$ is 
			$(x,0)$. In particular, we have $0=(0,0)$, $1=(1,0)$, and $(x,0)+(x',0)=(x+x',0)$
			for all $x,x'\in\bN$.
		\item We have $(x,0)+(0,y)=(x,y)$, for all $x,y\in\bN$.
		\item We have $(0,y)+(0,y')=(0,\max\{y,y'\})$
			and $(0,y)(0,y')=(0,\min\{y,y'\})$, for all $y,y'\in\bN$.
		\item We have $\infty+z=\infty$ for all $z\in W(\Bool)$.
		\item \label{part:W(Bool)-infty}
			We have $\infty\cdot\infty=\infty$ and $\infty\cdot(0,y)=(0,y)$ for all $y\in\bN$.
		\item For any nonzero partition $\lambda$, we have $m_\lambda(\infty)=\infty$ and
			$m_\lambda(0,y)=(0,[y/\lambda_1])$, where $[y/\lambda_1]$ denotes the
			greatest integer at most $y/\lambda_1$. 
		\item \label{part:W(Bool)-order}
			The bijection (\ref{map:W(Bool)-coords}) identifies the partial order $\wleq$ on
			$W(\Bool)$ with the componentwise partial order on $\bN^2\cup\{\infty\}$. That is,
			we have $(x_1,y_1)\wleq (x_2,y_2)$ if and only if $x_1\leq x_2$ and $y_1\leq y_2$,
			and $z\wleq\infty$ for all $z\in W(\Bool)$.
	\end{enumerate}
\end{proposition}
\begin{proof}

(1): Let $f$ denote the image of $x$ under the map $\bN\to W(\bN)$. Then we have
$\srp(f)=\srp(1)^x=(1+t)^x=1+\cdots+t^x$. Thus $f:\Lambda_\bN\to\bN$ satisfies
$f(m_{1^{x+1}})=f(e_{x+1})=0$ and $f(m_{1^x})=f(e_x)=1$. Therefore its image in $W(\Bool)$ has the same
properties; so its image is $(x,0)$.

(2): Let $I$ denote the partition kernel of the homomorphism $\Lambda_\bN\to\Bool$ given by
$(x,0)+(0,y)$. By the above, $I$ is empty or is of the form $I_{(x',y')}$. So it is
enough to show that the partition $(y+1)^{x+1}$ is an element of $I$ but that $y^{x+1}$ and
$(y+1)^x$ are not.

Recall that we have
	$$
	\Delta^+(m_\lambda) = \sum_{\mu,\nu} m_\mu \otimes m_\nu,
	$$
where $\mu$ and $\nu$ run over all partitions such that the vector
$(\mu_1,\nu_1,\mu_2,\nu_2,\dots)$ is a permutation of $\lambda$. Therefore $\lambda$ is in the
kernel of $(x,0)+(0,y)$ if and only if for all such decompositions of $\lambda$, we have either
$1^{x+1}\pleq\mu$ or $(y+1)^1\pleq\nu$.

Applying this to $\lambda=(y+1)^{x+1}$, we get terms of the form $m_{(y+1)^i} \otimes m_{(y+1)^j}$
where $i+j=x+1$. If $i>0$, then $(y+1)^i$ contains $(y+1)^1$. Otherwise, $j=x+1$, in which case
$(y+1)^{x+1}$ contains $1^{x+1}$. So $(y+1)^{x+1}$ is in the partition kernel $I$.

On the other hand, we can decompose $\lambda=y^{x+1}$ into $\mu=0$ and $\nu=y^{x+1}$. Then $\mu$
does not contain $1^{x+1}$ and $\nu$ does not contain $(y+1)^1$. So $y^{x+1}$ is not in $I$.
Similarly $(y+1)^x$ decomposes into $\mu=(y+1)^x$ and $\nu=0$, the first of which does not contain
$1^{x+1}$ and the second of which does not contain $(y+1)^1$. So $(y+1)^x$ is not in $I$.

(3): The image of any power sum $\psi_r=\sum_i x_i^r$
under $\Delta^+$ is $\psi_r\otimes 1 + 1\otimes
\psi_r$. Therefore the partition $r^1$ is in the partition kernel of $(0,y)+(0,y')$ if and only if
$r$ is greater than $y$ and $y'$. Therefore we have $(0,y)+(0,y')=(0,\max\{y,y'\})$. 

Similarly, we have $\Delta^\times(\psi_r)=\psi_r\tn\psi_r$, and so the partition $r^1$ is in the
partition kernel of $(0,y)(0,y')$ if and only if $r$ is greater than $y$ or $y'$. Therefore we have
$(0,y)(0,y')=(0,\min\{y,y'\})$.

(4): Since $\infty+z\geq \infty$, we have $\infty+z\wgeq\infty$, by~(\ref{pro:inequality-comparison}), and 
hence $\infty+z=\infty$.

(5): The map $\Delta^\times$ is injective because it has a retraction, for example the multiplicative
co-unit $\eps^\times$. Therefore for any partition $\lambda$, the tensor $\Delta^\times(m_\lambda)$ is a
nonempty sum of basic tensors $m_\mu\tn m_\nu$, all of which pair with $(\infty,\infty)$ to make $1$.
Therefore we have $(\infty\cdot\infty)(m_\lambda)=1$ for all $\lambda$, and hence
$\infty\cdot\infty=\infty$.

We have $\Delta^\times(\psi_r)=\psi_r\otimes \psi_r$ which pairs with $((0,y),\infty)$ to make
$(0,y)(\psi_r)$, which is $0$ if and only if $r>y$. Therefore the partition $r^1$ is in
$\pker(\infty\cdot (0,y))$ if and only if $r>y$. Thus we have $\pker(\infty\cdot (0,y))=I_{(0,y)}$, and hence
$(0,y)\cdot \infty=(0,y)$.

(6): The value of $m_\lambda(\infty)\in W(\Bool)=\Alg{\bN}(\Lambda_\bN,\Bool)$
at $m_\mu\in\Lambda_\bN$ is by definition the value of $\infty$ at
$m_\mu\circ m_\lambda$. Since $\lambda$ is nonzero, $m_\lambda$ is a sum of infinitely many
monomials in $x_1,x_2,\dots$. Therefore $m_\mu\circ m_\lambda$ is the sum of at least one such
monomial and is hence nonzero. It follows that $\infty(m_\mu\circ m_\lambda)=1$ for all
$\mu$ and hence $m_\lambda(\infty)=\infty$.

Similarly, the value of $m_\lambda(0,y)$ at $\psi_r$ is $(0,y)(\psi_r\circ m_\lambda)$. But this is
$1$ if and only if $r\lambda_1\leq y$. Therefore
$(m_\lambda(0,y))(\psi_r)$ is $1$ if and only if $r\leq[y/\lambda_1]$. And so we have
$m_\lambda(0,y)=(0,[y/\lambda_1])$.

(7): Both sides of the first part are clearly equivalent to $I_{(x_1,y_1)}\supseteq I_{(x_2,y_2)}$. 
The second part holds because $\infty$ corresponds to the empty partition ideal.
\end{proof}

\subsection{An interpretation of the algebraic structure on $W(\Bool)$}

We can express the description of $W(\Bool)$ given above in terms of the Dorroh extension, a general
construction in commutative algebra over $\bN$. Let us first recall it.

Let $f\:A\to B$ be a homomorphism of $\bN$-algebras. The Dorroh extension $D(f)$ is
defined to be $A\times B$ with the 
usual, product $\bN$-module structure and with a multiplication law given by the formula
	$$
	(x,y)(x',y') = (xx',f(x)y'+yf(x')+yy').
	$$
One can check that this is an $\bN$-algebra structure on $D(f)$ with multiplicative identity
$(1,0)$. For example, if $B$ is additively cancellative, then the map $D(f)\to A\times B$ sending 
$(x,y)$ to $(x,f(x)+y)$ is an injective homomorphism with image $\setof{(a,b)}{f(a)\leq b}$.

Now let $\bN'$ denote $\bN\cup\{\infty\}$ with the $\bN$-algebra structure where addition is $\max$
and multiplication is $\min$. Let $f$ denote the unique $\bN$-algebra map $\bN\to\bN'$. (So
$f(n)=\infty$ unless $n=0$.) Define a set map $g\:D(f)\to W(\Bool)$ by $(x,y)\mapsto (x,y)$ if
$y\neq\infty$ and $(x,\infty)\mapsto \infty$. Then $g$ is a surjective $\bN$-algebra morphism. It
follows that $W(\Bool)$ is the quotient of $D(f)$ obtained by identifying all elements of the form
$(x,\infty)$.

The algebra $\bN'$ can perhaps be demystified by observing that it is isomorphic to the image of the ghost
map $w\:W(\Bool)\to \Bool^\infty$. (See~(6.2.1) of~\cite{Borger:totalpos}.) 
Indeed, for a Witt vector of the form $(0,y)$, its image under the ghost map is the bit vector
$\langle{z_1,z_2,\dots}\rangle$ with $z_r=1$ if $r\leq y$ and otherwise $z_r=0$. Every other Witt vector
has image $\langle{1,1,\dots}\rangle$. Clearly addition and multiplication of such bit vectors are given by
taking the maximum and minimum of the corresponding $y$ values. One might say that $W(\Bool)$ can in
essence be recovered by applying a general construction to the image of the ghost map. It would be
interesting if something like this holds more generally.

It would also be interesting to give an explicit description of $W(A)$ as a $\Lambda_\bN$-semiring for any
$\Bool$-algebra $A$. This would be tantamount to having an explicit understanding of
multiplication, the coproducts $\Delta^+$ and $\Delta^\times$, and all right plethysm maps $f\mapsto f\circ
m_\lambda$ on the semiring $\Bool\tn_\bN\Lambda_\bN=\bigoplus_\lambda \Bool m_\lambda$ of symmetric functions
with Boolean coefficients.

\section{$\Wsch(\Bool)$}
\label{sec:Wsch-Bool}

The purpose of this section is to determine $\Wsch(\Bool)$. We will see in this section that it behaves very
similarly to $W(\Bool)$ in several ways. To stress these similarities, we are going to have the present
section mirror section~\ref{sect:Wbool} as closely as we can.

\subsection{Schur constituents and Schur prime ideals}
For any symmetric function $g\in\Sch$, let us say the \emph{(Schur) constituents} of $g$
are the partitions $\nu$ such that $g_\nu\neq 0$ in the unique decomposition 
$g=\sum_\nu g_\nu s_\nu$, $g_\nu\in\bN$, where the $s_\nu$ are the Schur symmetric functions.  

Let us say that a partition ideal $I$ is \emph{(Schur) prime} if $I\neq P$ and whenever partitions
$\lambda,\mu\in P$ have the property that every Schur constituent of $s_\lambda s_\mu$ is in $I$,
then either $\lambda$ or $\mu$ lies in $I$. It will follow from (\ref{pro:Wsch(Bool)-primeids})
below that Schur primality is equivalent to monomial primality; the distinction between the two
concepts is only needed until we prove that.

For any linear map $f\in\Mod{\bN}(\Sch,\Bool)$, write 
	$$
	\sker(f)=\setof{\lambda\in P}{f(s_\lambda)=0}.
	$$
We will call $\sker(f)$ the \emph{(Schur) partition kernel} of $f$.

\subsection{Constituents of products of Schur functions}

The combinatorial lemmas about monomial constituents of $m_{\lambda} m_{\mu}$ that came to our aid
in studying $W(\Bool)$ have Schur analogues. These analogues rest upon the
\emph{Littlewood--Richardson rule}, which is a formula for the coefficient of $s_\nu$ in the product
$s_{\lambda} s_{\mu}$. In order to be clear about the form of the rule we will use, we will fix some
terms. For the basic terminology of tableaux, see Stanley's book~\cite{Stanley:EC2}
(section 7.10, p.\ 309).

The \emph{reverse reading word} of a tableau $T$ is defined to be the word obtained as follows: write down the first row of $T$ in reverse order (that is, from right to left); after it, write down the second row of $T$ in reverse order; after it, the third one, and so on.

A word $w$ whose letters are positive integers is said to be a \emph{lattice permutation} if and only if every positive integer $i$ and every $j\in\left\lbrace 0,1,...,n\right\rbrace$ (where $n$ is the length of $w$) satisfy the following condition: The number of $i$'s among the first $j$ letters of $w$ is at least as large as the number of $i+1$'s among the first $j$ letters of $w$.

A skew semistandard tableau $T$ is said to be a \emph{Littlewood--Richardson tableau} if and only if its reverse reading word is a lattice permutation. 

If a tableau $T$ is filled with positive integers, then the \emph{type} of the tableau $T$ is defined to be the sequence $\left(n_1,n_2,n_3,...\right)$, where $n_k$ is the number of boxes of $T$ filled with the integer $k$. Note that the type of a tableau $T$ is not necessarily a partition, but it is a partition whenever $T$ is a Littlewood--Richardson tableau.

For any partitions $\lambda$, $\mu$ and $\nu$, we let $c_{\lambda\mu}^{\nu}$ denote the number of Littlewood--Richardson tableaux of shape $\nu / \lambda$ and type $\mu$. When $\lambda\pnleq\nu$, this number is understood to be $0$. It is easy to see that $c_{\lambda\mu}^{\nu} = 0$ unless $\left|\lambda\right| + \left|\mu\right| = \left|\nu\right|$. It is much less obvious that $c_{\lambda\mu}^{\nu} = c_{\mu\lambda}^{\nu}$ for all partitions $\lambda$, $\mu$ and $\nu$. The numbers $c_{\lambda\mu}^{\nu}$ are called the \emph{Littlewood--Richardson coefficients}.

The Littlewood--Richardson rule then states
	\begin{equation}
	\label{eq:WSch(Bool)-lr}
	s_{\lambda} s_{\mu} = \sum\limits_{\nu\in P} c_{\lambda\mu}^{\nu} s_{\nu}.
	\end{equation}
for any partitions $\lambda,\mu$. One can find a proof in Stanley's book~\cite{Stanley:EC2}
(theorem A1.3.3, p.\ 432, appendix to Chapter 7). For other points of view, see
Sagan's book~\cite{Sagan:Symmetric-group} (theorem 4.9.4) or van~Leeuwen~\cite{vanLeeuwen:LR-rule} 
(1.4.4--1.4.5).

The following analogue of (\ref{pro:W(Bool)-comblem1}) is then clear.

\begin{proposition}
\label{pro:Wsch(Bool)-comblem1}
Let $\lambda$ and $\mu$ be two partitions. Then the Schur constituents of the product $s_{\lambda} s_{\mu}$ are precisely those partitions $\nu$ for which there exists a Littlewood--Richardson tableau of shape $\nu / \lambda$ and type $\mu$.
\end{proposition}

A corollary analogous to  (\ref{cor:W(Bool)-comblem2}) results:

\begin{corollary}\label{cor:Wsch(Bool)-comblem2}
Let $\lambda$ and $\mu$ be two partitions. Then, every Schur constituent $\nu$ of $s_{\lambda} s_{\mu}$ satisfies $\lambda\pleq \nu$ and $\mu\pleq \nu$.
\end{corollary}
\begin{proof}
Let $\nu$ be a Schur constituent of $s_{\lambda} s_{\mu}$. Due to (\ref{pro:Wsch(Bool)-comblem1}), there exists a Littlewood--Richardson tableau of shape $\nu / \lambda$ and type $\mu$. As a consequence, the shape $\nu / \lambda$ must be well-defined, so that $\lambda\pleq \nu$. Similarly, $\mu\pleq \nu$.
\end{proof}

There is also a counterpart of (\ref{pro:W(Bool)-comblem3}) in the Schur function
setting:

\begin{proposition}\label{pro:Wsch(Bool)-comblem3}
Let $\lambda$ and $\mu$ be two partitions, and let $i$ and $j$ be two positive integers such that $\lambda_i > \mu_i$, $\lambda_i > \lambda_{i+1}$, $\mu_j > \lambda_j$, and $\mu_j > \mu_{j+1}$. Let $\bar{\lambda}$ denote the result of diminishing the $i$-th part of $\lambda$ by $1$, and let $\bar{\mu}$ be the result of diminishing the $j$-th part of $\mu$ by $1$. Then, for every Schur constituent $\nu$ of $s_{\bar{\lambda}} s_{\bar{\mu}}$, we have either $\lambda\pleq\nu$ or $\mu\pleq\nu$.
\end{proposition}

\begin{proof}
Since $\lambda_i > \mu_i$ and $\mu_j > \lambda_j$, we have $i\neq j$. Thus, we can assume by symmetry that $i<j$.

For every partition $\gamma$, let $\gamma_{\setminus i}$ denote the result of removing the $i$-th
part from the partition $\gamma$. By the definition of $\bar{\lambda}$, we have
$\bar{\lambda}_{\setminus i} = \lambda_{\setminus i}$.

Fix a Schur constituent $\nu$ of $s_{\bar{\lambda}} s_{\bar{\mu}}$. Then, by 
(\ref{pro:Wsch(Bool)-comblem1}), there exists a Littlewood--Richardson tableau of shape $\nu /
\bar{\lambda}$ and type $\bar{\mu}$. Denote this tableau by $T$. Since $T$ exists, we have
$\bar{\lambda} \pleq \nu$. In particular, $\nu_i \geq \bar{\lambda}_i \geq \mu_i \geq \mu_j$ (since
$i < j$).

First suppose that the $i$-th row of $\nu / \bar{\lambda}$ is nonempty. Then we have $\nu_i \geq
\bar{\lambda}_i + 1 = \lambda_i$. Since $\nu$ is a Schur constituent of $s_{\bar{\lambda}}
s_{\bar{\mu}}$, we have by  (\ref{cor:Wsch(Bool)-comblem2}) that $\bar{\lambda} \pleq
\nu$. Therefore we have $\lambda\pleq\nu$, which finishes the proof in this case.

Now suppose that the $i$-th row of $\nu / \bar{\lambda}$ is empty. In this case, we will show $\mu
\pleq \nu$. Since we have $\bar{\mu} \pleq \nu$, as above, it is enough to show $\nu_j\geq \mu_j$.

We can obtain a new tableau $T'$ from the tableau $T$ by removing its (empty) $i$-th row and then
moving each row below the $i$-th one up by one unit of length. This new tableau $T'$ is a
Littlewood--Richardson tableau of shape $\nu_{\setminus i} / \bar{\lambda}_{\setminus i}$ and type
$\bar{\mu}$ (since the original tableau $T$ was a Littlewood--Richardson tableau of shape $\nu /
\bar{\lambda}$ and type $\bar{\mu}$). Hence, there exists a Littlewood--Richardson tableau of shape
$\nu_{\setminus i} / \bar{\lambda}_{\setminus i}$ and type $\bar{\mu}$. By
(\ref{pro:Wsch(Bool)-comblem1}) (applied to $\bar{\lambda}_{\setminus i}$, $\bar{\mu}$ and
$\nu_{\setminus i}$ instead of $\lambda$, $\mu$ and $\nu$), this yields that $\nu_{\setminus i}$ is
a Schur constituent of the product $s_{\bar{\lambda}_{\setminus i}} s_{\bar{\mu}}$. By
(\ref{cor:Wsch(Bool)-comblem2}) (again applied to $\bar{\lambda}_{\setminus i}$, $\bar{\mu}$ and
$\nu_{\setminus i}$ instead of $\lambda$, $\mu$ and $\nu$), this yields that
$\bar{\lambda}_{\setminus i} \pleq \nu_{\setminus i}$ and $\bar{\mu} \pleq \nu_{\setminus i}$. But
the partition $\bar{\mu}$ has at least $j-1$ parts at least $\mu_j$ (namely, its first
$j-1$ parts, which as we know are unchanged from $\mu$). Since $\bar{\mu} \pleq \nu_{\setminus i}$,
this implies that the partition $\nu_{\setminus i}$ also has at least $j-1$ parts at least
$\mu_j$. Hence, the partition $\nu$ has at least $j$ parts at least $\mu_j$ (because
it contains all the parts of $\nu_{\setminus i}$, along with $\nu_i$ which as we know is at least
$\mu_j$). In other words, $\nu_j\geq \mu_j$.
\end{proof}

We will need two more facts about Schur constituents. The following is an analogue of
(\ref{pro:W(Bool)-lr-extremals}). It will also be used only in section~\ref{sec:countability}.

\begin{proposition}\label{pro:Wsch(Bool)-lr-extremals}
Let $\lambda$ and $\mu$ be two partitions.
	\begin{enumerate}
		\item $\lambda + \mu$ is a Schur constituent of $s_\lambda s_\mu$.
		\item $\lambda \uplus \mu$ is a Schur constituent of $s_\lambda s_\mu$,
			where $\uplus$ is as in~(\ref{pro:W(Bool)-lr-extremals}).
		\item If $\lambda$ and $\mu$ are both nonzero, then $s_\lambda s_\mu$ has at least two distinct nonzero Schur constituents.
	\end{enumerate}
\end{proposition}

\begin{proof}
(1): By (\ref{pro:Wsch(Bool)-comblem1}), it is enough to show that there exists a
Littlewood--Richardson tableau of shape $\left(\lambda + \mu\right) / \lambda$ and type $\mu$. Such
a tableau can be made by filling the $i$-th row of the diagram $\left(\lambda + \mu\right) /
\lambda$ with the number $i$ for every positive integer $i$.

(2): For any partition $\gamma$, let $\gamma'$ denote the conjugate partition. Since there is a
$\bZ$-algebra isomorphism $\omega : \Lambda_{\bZ}\to\Lambda_{\bZ}$ satisfying
$\omega\left(s_\gamma\right) = s_{\gamma'}$ for every partition $\gamma$, it is enough to check
that $\left(\lambda \uplus \mu\right)'$ is a Schur constituent of $s_{\lambda'} s_{\mu'}$. But
since we have $\lambda' + \mu' = \left(\lambda \uplus \mu\right)'$, we are done by part (1).

(3): Since $\lambda$ and $\mu$ are nonzero, $\lambda+\mu$ and $\lambda\uplus\mu$ are
distinct and nonzero. By parts (1) and (2), they are also Schur constituents of $s_\lambda s_\mu$.
\end{proof}

\subsection{Remark} When $\lambda$ and $\mu$ are two partitions, the two partitions $\lambda + \mu$ and $\lambda \uplus \mu$ are, respectively, the highest and the lowest Schur constituents of $s_\lambda s_\mu$ with respect to the so-called dominance order on partitions (also known as the majorization order)---the order in which a partition $\gamma$ is at least a partition $\delta$ of the same size if and only if every positive integer $i$ satisfies $\gamma_1+\gamma_2+...+\gamma_i \geq \delta_1+\delta_2+...+\delta_i$. 

\begin{proposition}\label{pro:Wsch(Bool)-comblem4}
Let $\lambda$ and $\mu$ be two partitions. Let $x$ be a positive integer. Let $\nu$ denote the partition obtained by sorting the
components of the vector
	$$
	(\lambda_1+\mu_1,\dots,\lambda_{x}+\mu_x,\lambda_{x+1},\mu_{x+1},\lambda_{x+2},\mu_{x+2},\dots)
	$$
in weakly decreasing order. Then $\nu$ is a Schur constituent of $s_\lambda s_\mu$.
\end{proposition}

\begin{proof}
According to  (\ref{pro:Wsch(Bool)-comblem1}), it is enough to show that there exists a Littlewood--Richardson tableau of shape $\nu / \lambda$ and type $\mu$. We are going to construct such a tableau $T$.

Indeed, we are going to give two ways to construct such a $T$ -- one explicit way which we will only sketch, and one less explicit one which we will detail.

\emph{First construction of $T$:} Let us first notice that $\nu_i = \lambda_i + \mu_i$ for every $i\in\left\lbrace 1,2,...,x\right\rbrace$, so that the $i$-th row of $\nu / \lambda$ has length $\mu_i$ for every $i\in\left\lbrace 1,2,...,x\right\rbrace$. Let us fill this row with $i$'s. Thus, the first $x$ rows of $\nu / \lambda$ are filled.

Now, for every cell $\left(i,j\right)$ of $\nu / \lambda$ with $i > x$, we fill the cell $\left(i,j\right)$ with $x$ plus the number of all $u \in \left\lbrace x+1,x+2,...,i \right\rbrace$ satisfying $\lambda_u < j$.

This gives us a tableau $T$ of shape $\nu / \lambda$, which the reader can verify to be a Littlewood--Richardson tableau of type $\mu$ (the verification is not trivial, but does not require any new ideas). This proves the existence of such a tableau and thus concludes the proof of  (\ref{pro:Wsch(Bool)-comblem4}).

\emph{Second construction of $T$:} For every partition $\gamma$ and any integer $y$, let $\gamma_{\leq y}$ denote the partition $\left(\gamma_1,\gamma_2,...,\gamma_y\right)$, and let $\gamma_{>y}$ denote the partition $\left(\gamma_{y+1},\gamma_{y+2},\dots\right)$. Clearly, $\nu_{\leq x} = \lambda_{\leq x} + \mu_{\leq x}$ and $\nu_{>x} = \lambda_{>x} \uplus \mu_{>x}$, where the operators $+$ and $\uplus$ on partitions are defined as in  (\ref{pro:Wsch(Bool)-lr-extremals}).

Since $\nu_{\leq x} = \lambda_{\leq x} + \mu_{\leq x}$, the result
(\ref{pro:Wsch(Bool)-lr-extremals})(1) (applied to $\lambda_{\leq x}$ and $\mu_{\leq x}$ instead of
$\lambda$ and $\mu$) implies that $\nu_{\leq x}$ is a Schur constituent of $s_{\lambda_{\leq x}}
s_{\mu_{\leq x}}$. Hence, by  (\ref{pro:Wsch(Bool)-comblem1}), there exists a
Littlewood--Richardson tableau of shape $\nu_{\leq x} / \lambda_{\leq x}$ and type $\mu_{\leq x}$.
Let $T_1$ be such a tableau.

Since $\nu_{>x} = \lambda_{>x} \uplus \mu_{>x}$, the result (\ref{pro:Wsch(Bool)-lr-extremals})(2) (applied to $\lambda_{>x}$ and $\mu_{>x}$ instead of $\lambda$ and $\mu$) implies that $\nu_{>x}$ is a Schur constituent of $s_{\lambda_{>x}} s_{\mu_{>x}}$. Hence, by  (\ref{pro:Wsch(Bool)-comblem1}), there exists a Littlewood--Richardson tableau of shape $\nu_{>x} / \lambda_{>x}$ and type $\mu_{>x}$. Let $T_2$ be such a tableau. Denote by $\widehat{T_2}$ the result of moving the tableau $T_2$ south (i.e., down) by $x$ rows and adding $x$ to each of its entries. It is easy to see that $\widehat{T_2}$ is a Littlewood--Richardson tableau of shape $\nu / \left(\nu_1,\nu_2,...,\nu_x,\lambda_{x+1},\lambda_{x+2},\lambda_{x+3},...\right)$. Hence, $\widehat{T_2}$ doesn't intersect $T_1$, and moreover, each cell of $\widehat{T_2}$ lies strictly southwest of each cell of $T_1$. Thus, it is rather clear that overlaying $T_1$ with $\widehat{T_2}$ gives us a Littlewood--Richardson tableau
of shape $\nu / \lambda$ and type $\mu$. This once again proves that such a tableau exists, and as we know this establishes  (\ref{pro:Wsch(Bool)-comblem4}).
\end{proof}

\begin{proposition}\label{pro:Wsch(Bool)-primeparid}
Let $2^P$ denote the set of subsets of $P$.
	\begin{enumerate}
		\item The map $f\mapsto \sker(f)$ is a bijection $\Mod{\bN}(\Sch,\Bool)\to 2^P$.
		\item $\sker(f)$ is a partition ideal if and only if $f$ satisfies
			$f(x)=0 \Rightarrow f(xy)=0$.
		\item \label{part:schur-prime-char}
			$\sker(f)$ is a Schur prime partition ideal if and only if $f$ is an $\bN$-algebra 
			map.
	\end{enumerate}
\end{proposition}
\begin{proof}
Proceed as in (\ref{pro:W(Bool)-primeparid}), but invoke (\ref{pro:Wsch(Bool)-comblem1}) and 
(\ref{cor:Wsch(Bool)-comblem2}) instead of  (\ref{pro:W(Bool)-comblem1}) and  (\ref{cor:W(Bool)-comblem2}). (Notice that the $a_{\mu}$ in the proof of part (2) are now all equal to $1$ according to the Pieri rule.)
\end{proof}

It turns out that the Schur prime partition ideals are exactly the monomial prime partition ideals. Indeed, we have the following analogue of  (\ref{pro:W(Bool)-primeids}):

\begin{proposition}\label{pro:Wsch(Bool)-primeids}
For integers $x,y\in\bN$, consider the partition ideal
	$$
	I_{(x,y)}=\setof{\lambda\in P}{(y+1)^{x+1}\pleq \lambda}
	$$
as defined in  (\ref{pro:W(Bool)-primeids}).
Then $I_{(x,y)}$ is Schur prime. Conversely, every nonempty Schur prime partition ideal is of this form.
\end{proposition}
\begin{proof}
This proof proceeds by adapting the proof of  (\ref{pro:W(Bool)-primeids}). The changes that are required (like replacing ``prime'' by ``Schur prime'', and referring to  (\ref{pro:Wsch(Bool)-comblem3}) and  (\ref{cor:Wsch(Bool)-comblem2}) instead of  (\ref{pro:W(Bool)-comblem3}) and  (\ref{cor:W(Bool)-comblem2})) are almost all obvious; the only nontrivial change is to replace the reference to  (\ref{pro:W(Bool)-comblem1}) by a reference to  (\ref{pro:Wsch(Bool)-comblem4}).
\end{proof}

\begin{proposition}
\label{pro:Wsch(Bool)-arithmetic}
Let $\infty$ denote the element of $\Wsch(\Bool)$ satisfying $\infty(s_\lambda)=1$ for all 
partitions $\lambda$, and let $\anti{1}\in\Wsch(\Bool)$ be as in~(\ref{eq:anti}).
Then we have
\begin{enumerate}
	\item
		\label{part:Wsch(Bool)-addition}
		$\sker(x+\anti{1}y)=I_{(x,y)}$, for all $x,y\in\bN$
	\item $\infty+z=\infty$, for all $z\in\Wsch(\Bool)$
	\item $\infty\cdot z=\infty$, for all nonzero $z\in\Wsch(\Bool)$
	\item $s_\lambda(\infty)=\infty$, for all nonzero partitions $\lambda$.
\end{enumerate}
\end{proposition}
\begin{proof}
(1): For clarity, given any $z\in \bN$, let us write $f_z\:\Sch\to\Bool$ for the image of $z$ 
under the unique semiring map $\bN\to\Wsch(\Bool)$. So if we let $I$ denote the partition ideal 
$\sker(f_x+\anti{1}f_y)$, we need to show $I=I_{(x,y)}$.

In the case where $y=0$, the argument is identical to that 
for~(\ref{pro:W(Bool)-description})(\ref{part:N-to-W(Bool)}).
On the other hand, if $x=0$, then we have $(\anti{1}f_y)(s_\lambda) = f_y(s_{\lambda'})$,
where $\lambda'$ is the conjugate partition of $\lambda$. (See for example~\cite{Borger:totalpos},~(6.11).)
And so this case follows from the previous one.

Now consider the general case.
By (\ref{pro:Wsch(Bool)-primeparid}) and (\ref{pro:Wsch(Bool)-primeids}), 
$I$ is either empty or of the form $I_{(x',y')}$. 
Therefore it is enough to show the three statements
	\begin{equation}
	\label{eq:local-23095}
	(y+1)^{x+1}\in I, \quad y^{x+1}\not\in I, \quad (y+1)^x\not\in I.
	\end{equation}
Since we have
	$$
	\Delta^+(s_\lambda)=\sum_{\mu,\nu}c^\lambda_{\mu\nu}s_\mu\tn s_\nu,
	$$
by the definition of addition of Witt vectors, we have
	$$
	(f_x+\anti{1}f_y)(s_\lambda) = \sum_{\mu,\nu}c^\lambda_{\mu\nu}\cdot 
	f_x(s_\mu)\cdot(\anti{1}f_y)(s_\nu).
	$$
Therefore a partition $\lambda$ is in $I$ if and only if we have 
$\mu\in\pker(f_x)$ whenever $c^\lambda_{\mu\nu}\geq 1$ and $\nu\not\in\pker(\anti{1}f_y)$.
Combining this with the two initial cases above we see that $\lambda$ is in $I$ if and only if 
the following implication holds:
	$$
	\big(c^\lambda_{\mu\nu}\geq 1 \text{ and } \nu_1\leq y\big) \Longrightarrow \mu_{x+1}\geq 1.
	$$
We will now use this criterion to show each of the three statements in~(\ref{eq:local-23095}).

First consider $\lambda=(y+1)^{x+1}$, and assume $c^\lambda_{\mu\nu}\geq 1$ and $\nu_1\leq y$. Then
there is a Littlewood--Richardson tableau of shape $\lambda/\nu$ and type $\mu$.
Since $\nu_1\leq y$, the diagram of
$\lambda/\nu$ contains the entire rightmost column of $\lambda$. Therefore all entries in that
column of the tableau must be distinct, and hence $\mu$ has length at least $x+1$.

Now consider $\lambda=y^{x+1}$, and set $\nu=\lambda$ and $\mu=0$. Then there is clearly a
Littlewood--Richardson tableau of shape $\lambda/\nu$ and type $\mu$, and further we have $\nu_1\leq y$ but
$\mu_{x+1}=0$. Therefore $y^{x+1}\not\in I$. Similarly, for $\lambda=(y+1)^x$, take $\nu=0$ and
$\mu=\lambda$. We have $\nu_1\leq y$, and there is clearly a Littlewood--Richardson tableau of shape
$\lambda/\nu=\lambda$ and type $\mu=\lambda$, but $\mu_{x+1}=0$. Therefore $(y+1)^x\not\in I$.

(2): Since $\infty+z\geq \infty$, we have $\infty+z\wgeq\infty$, by~(\ref{pro:inequality-comparison}), and 
hence $\infty+z=\infty$.

(3): By distributivity and parts (1) and (2), it is enough to check this for the $\bN$-module generators
$z=1,\anti{1},\infty$. It is clear for $z=1$. For $z=\anti{1}$, it holds because multiplication by
$\anti{1}$ is an involution which preserves $\Wsch(\Bool)\setminus\{\infty\}$, by part (1), and
hence fixes $\infty$. When $z=\infty$, the argument is that of
(\ref{pro:W(Bool)-description})(\ref{part:W(Bool)-infty}) but with the Schur basis instead of
the monomial one.

(4): By definition $\big(s_\lambda(\infty)\big)(s_\mu)=\infty(s_\mu\circ s_\lambda)$. So the
statement to be shown is equivalent to the statement $s_\mu\circ s_\lambda\neq 0$ for all partitions 
$\mu$. Since $\lambda\neq 0$, we can write $s_\lambda=y_1+y_2+\cdots$,
where each $y_i$ is a monomial in the formal variables $x_1,x_2,\dots$ with coefficient $1$. Therefore we have
$s_\mu\circ s_\lambda=s_\mu(y_1,y_2,\dots)$, but this is nonzero since all Schur functions are nonzero.
\end{proof}

\subsection{Explicit description of $\Wsch(\Bool)$}
\label{subsec:explicit-Wsch(Bool)}
Here we prove part~(\ref{part:Wsch(Bool)-semiring}) of theorem~(\ref{thm:intro-B}), although we use slightly different notation. It follows from~(\ref{pro:Wsch(Bool)-primeparid}), (\ref{pro:Wsch(Bool)-primeids}), and (\ref{pro:Wsch(Bool)-arithmetic})(\ref{part:Wsch(Bool)-addition}) that we have a bijection
	\begin{equation}
	\label{eq:Wsch(Bool)-explicit}
	\bN[\eta]/(\eta^2=1)\cup\{\infty\} \longisomap \Wsch(\Bool)
	\end{equation}
sending $x+y\eta$ to $x+\anti{1}y$ and sending $\infty$ to $\infty$. It is an isomorphism of $\bN$-algebras
if we give the left-hand side the $\bN$-algebra structure extending that on $\bN[\eta]/(\eta^2=1)$ and satisfying 
$\infty+z=\infty$ and $\infty\cdot z=\infty$, for $z\neq 0$. This follows 
from~(\ref{pro:Wsch(Bool)-arithmetic}) and the equality $\anti{1}^2=1$.
It is also clearly order preserving, as in~(\ref{pro:W(Bool)-description})(\ref{part:W(Bool)-order}).
The $\Sch$-semiring structure on $\bN[\eta]/(\eta^2=1)\cup\{\infty\}$ induced by this isomorphism
is the unique one satisfying $s_\lambda(\infty)=\infty$ and 
	$$
	s_\lambda(\eta) = 
		\begin{cases}
			\eta^r & \text{if }\lambda=(1^r) \\ 
			0 & \text{otherwise.}
		\end{cases}		
	$$

\begin{proposition}
\label{prop:Bool-map}
In terms of the bijections~(\ref{map:W(Bool)-coords}) and~(\ref{eq:Wsch(Bool)-explicit}), the map $W(\Bool)\to\Wsch(\Bool)$ sends any element of the form $(x,0)$ 
to $x=x+\anti{1}0$. It sends all other elements to $\infty$.	
\end{proposition}

\begin{proof}
By (\ref{pro:W(Bool)-description})(\ref{part:N-to-W(Bool)}), the pair $(x,0)$ is the image of $x$ under the unique map $\bN\to W(\Bool)$; so it must map to $x$ in $\Wsch(\Bool)$. Therefore all that remains is to show that the pre-image of $\infty$ under the map $W(\Bool)\to\Wsch(\Bool)$ contains all other elements. By algebraic laws listed in theorem~(\ref{thm:intro-B}), it is enough to show the pre-image of $\infty$ contains $\infty$ and all elements of the form $(0,y)$, where $y\geq 1$.

Let $\lambda$ be a partition, and put $r=\sum_i \lambda_i$. Then it is easily seen that there is at least
one semistandard Young tableau of shape $\lambda$ and type $1^r$. So the Kostka number $K_{\lambda,1^r}$ is
at least $1$. Therefore we have $s_\lambda=m_{1^r}+\cdots \geq m_{1^r}$, and hence
	$
	(0,y)(s_\lambda)\geq (0,y)(m_{1^r})=1
	$
for any $y\geq 1$. Therefore any such $(0,y)$ maps to $\infty$ in $\Wsch(\Bool)$. Last, $\infty$ also maps
to $\infty$, because the map $\infty\:\Lambda_\bN\to\Bool$ is $1$ on all nonzero elements of $\Lambda_\bN$
and therefore on all Schur functions $s_\lambda$.
\end{proof}

\subsection{Remark}
Observe that while the map $W(A)\to\Wsch(A)$ is injective when $A$ has additive cancellation, it is not so in general. Indeed, this fails for $A=\Bool$.

\section{Total positivity}
\label{sec:total-pos}
In this section, we recall the relationship between the Schur Witt vectors and Schoenberg's theory of total positivity. It gives another technique for understanding Witt vectors of $\bRpl$ and hence of semirings that map to $\bRpl$. All details can be found in section 7 of~\cite{Borger:totalpos}.

\subsection{Total positivity}
As in the introduction, consider the map
	\begin{align}
	\begin{split}
		\label{map:ser++}
		\srp\: \Wsch(A) &\longmap 1+tA[[t]] \\
		a &\longmapsto 1+a(e_1)t+a(e_2)t^2+\cdots,
	\end{split}
	\end{align}
where the $e_n$ are the elementary symmetric functions
	$$
	e_n = \sum_{i_1<\dots< i_n} x_{i_1}\cdots x_{i_n} \in\Sch.
	$$
It is a bijection if $A$ is a ring and an injection if $A$ has additive cancellation.

Following Schoenberg~\cite{Schoenberg:Courant-birthday}, one says that a formal real series $f(t)=\sum_{i\in\bZ} a_i t^i$ is totally positive if all the minors of the infinite matrix $(a_{i-j})_{i,j}$ are $\geq 0$. The connection between total positivity and $\Wsch$ is given by two standard facts in the theory of symmetric functions, that $\Sch$ agrees with the $\bN$-linear span of the skew Schur functions and that the skew Schur functions are the minors of the matrix $(e_{i-j})_{i,j}$. It follows that for any subring $B\subseteq\bR$, the map (\ref{map:ser++}) restricts to a bijection
	$$
	\srp\:\Wsch(B\cap\bRpl) \longisomap \big\{\text{totally positive series in }1+tB[[t]]\big\}.
	$$
The classification of the totally positive series given by the Edrei--Thoma theorem then amounts to the following:

\begin{theorem}
\label{thm:Witt-E-T}
	The map $\srp\:W(\bR)\to 1+t\bR[[t]]$ restricts to a bijection
		\begin{equation}
			\label{map:E-T}
		\Wsch(\bRpl) \longisomap \bigsetof{e^{\gamma t}\frac{\prod_{i=1}^\infty(1+\alpha_i t)}{\prod_{i=1}^\infty(1-\beta_i t)}}{\alpha_i,\beta_i,\gamma\in\bRpl, \sum_i(\alpha_i+\beta_i)<\infty}.			
		\end{equation}
\end{theorem}

This was conjectured by Schoenberg~\cite{Schoenberg:Courant-birthday} and proved by Aissen--Schoenberg--Whitney~\cite{ASW} and, in the final step, Edrei~\cite{AESW}\cite{Edrei:totally-positive}. It was independently discovered and proved by Thoma in his work~\cite{Thoma:totally-positive} on the asymptotic character theory of symmetric groups. Both arguments use hard results in the theory of entire functions. Later, proofs entirely within asymptotic character theory emerged out of the work of Vershik, Kerov, Okounkov, and Olshanski~\cite{Vershik-Kerov:asymptotic}\cite{Kerov:book}\cite{Okounkov:InfSymGroup}\cite{Kerov-Okounkov-Olshanski}.

The following is an easy consequence of (\ref{thm:Witt-E-T}) (see section 7 of~\cite{Borger:totalpos}):

\begin{corollary}\label{cor:totalpos-summary}
	The map $\srp\:W(\bR)\to 1+t\bR[[t]]$ restricts to bijections
		\begin{align*}
			W(\bRpl) &\longisomap \bigsetof{e^{\gamma t}\prod_{i=1}^\infty(1+\alpha_i t)}{\alpha_i,\gamma\in\bRpl,\sum_i\alpha_i<\infty} \\
			W(\bN) &\longisomap \bigsetof{f(t)\in 1+t\bZ[t]}{\text{all complex roots of $f(t)$ are
				negative reals}}.
		\end{align*}
\end{corollary}

The algebra $\Wsch(\bN)$ is determined in the following section.

\section{Edrei--Thoma and Boolean Witt vectors}
\label{sec:transition-morphisms}

The purpose of this section is to describe the map on Witt vectors induced by the map $\bRpl\to\Bool$ in terms of the power-series description of $W(\bRpl)$ and the combinatorial description of $W(\Bool)$. We then use this to prove theorem~(\ref{thm:intro-A}), or equivalently~(\ref{thm:Davydov}).

\begin{proposition}
\label{thm:intro-C}
\begin{enumerate}
	\item \label{part:monomial-disc-inv}
		In terms of~(\ref{cor:totalpos-summary}) and~(\ref{thm:intro-B}), 
		the functorial map $W(\bRpl)\to W(\Bool)$ satisfies
			\begin{equation}
			e^{\gamma t} f(t) \longmapsto
				\begin{cases}
				(\deg f(t),0) & \text{if }\gamma=0 \\ 
				(\deg f(t),1) & \text{if }\gamma\neq0, 
				\end{cases}		
			\end{equation}
		for any polynomial $f(t)$. It sends all other series to $\infty$.
	\item \label{part:disc-inv}
		In terms of~(\ref{thm:Witt-E-T}) and~(\ref{thm:intro-B}), the functorial map 
		$\Wsch(\bRpl)\to\Wsch(\Bool)$ sends any rational function $f(t)/g(t)$ in lowest terms
		to the pair $(\deg f(t), \deg g(t))$. It sends all other series to $\infty$. 
\end{enumerate}
\end{proposition}
\begin{proof}
(\ref{part:monomial-disc-inv}): 
Let $\varphi$ denote the functorial map  $\varphi\:W(\bRpl)\to W(\Bool)$.
Let $z\in W(\bRpl)$ be a Witt vector. By~(\ref{cor:totalpos-summary}), we can write
	$$
	\srp(z)=e^{\gamma t}\prod_{i=1}^\infty (1+\alpha_i t).
	$$
Let $r\leq\infty$ be the number of nonzero $\alpha_i$. Abusing notation, let us still write $r$ for the
image of $r$ under the map $\bN\cup\{\infty\}\to W(\Bool)$. (So $r=(r,0)$ for $r\neq\infty$.)
By~(\ref{pro:W(Bool)-description}), it is enough to show $\varphi(z)=r$ if $\gamma=0$, and
$\varphi(z)=r+(0,1)$ if $\gamma\neq 0$. It is also enough to consider
separately the cases where $\gamma=0$ and where all $\alpha_i=0$; this is because
the map $\srp$ identifies addition of Witt vectors with multiplication of power 
series.

First suppose $\gamma=0$. Then we have $z(m_\lambda)=m_\lambda(\alpha_1,\alpha_2,\dots)$ for all
partitions $\lambda$. (For example, one can observe that $z$ is the image of $\sum_i[\alpha_i]$ under the
map $\Nhat{\bRpl}\to W(\bRpl)$ from the convergent monoid algebra. See~(7.6) 
and~(4.12.1) in~\cite{Borger:totalpos}.) If $r=\infty$, this is zero for no 
$\lambda$, and if $r<\infty$, it is zero if and only if $\lambda_{r+1}\geq 1$. Therefore the partition
ideal corresponding to $\varphi(z)$ is the empty one if $r=\infty$ and is $I_{(r,0)}$ if $r<\infty$. Hence
we have $\varphi(z)=r$, whether $r$ is infinite or not.

Now suppose $\gamma\neq 0$ and $\alpha_i=0$ for all $i$. Then we have
	$$
	\exp(\gamma t)=\srp(z)=\exp\Big(-\sum_{n\geq 1}z(\psi_n)\frac{(-t)^n}{n}\Big),
	$$
and hence $z(m_{1^1})=z(\psi_1)=\gamma\neq 0$ and $z(m_{2^1})=z(\psi_2)=0$. Therefore $\varphi(z)$ 
has partition kernel $I_{(0,1)}$, and so $\varphi(z)=(0,1)$.

(\ref{part:disc-inv}): 
Let $\phiSch$ denote the functorial map $\Wsch(\bRpl)\to\Wsch(\Bool)$.
By~(\ref{thm:Witt-E-T}) and~(\ref{cor:totalpos-summary}), any element $z\in\Wsch(\bRpl)$ is of
the form $x+\anti{1}y$ for some $x,y\in W(\bRpl)$. Because anti-Teichm\"uller lifts are functorial,
we have $\phiSch(x+\anti{1}y) = \phiSch(x)+\anti{1}\phiSch(y)$. Therefore 
by~(\ref{pro:Wsch(Bool)-arithmetic}), we are reduced to the case $z\in W(\bRpl)$, and this follows from
(\ref{prop:Bool-map}) and part (\ref{part:monomial-disc-inv}) above.
\end{proof}

\subsection{Remark}
The Edrei--Thoma theorem~(\ref{thm:Witt-E-T}) provides a discrete invariant associated to a totally positive
series, namely the pair consisting of the number of zeros and the number of poles, or $\infty$ if the series is
not a rational function. By (\ref{thm:intro-C})(\ref{part:disc-inv}), we may interpret this invariant as the
image of the corresponding Witt vector under the natural map $\Wsch(\bRpl)\to\Wsch(\Bool)$. So we have the
satisfying fact that the coarse invariant coming from the Boolean Witt vectors discussed in the introduction
agrees with the obvious discrete invariant given by the Edrei--Thoma theorem.

\subsection{Remark}
The map $\Wsch(\bRpl)\to \Wsch(\Bool)$ is surjective, but $W(\bRpl)\to W(\Bool)$ is not.

\begin{lemma}
\label{lem:goes-extinct}
Given any $a\in \Wsch(\bRpl)$, we have
	$
	\lim_{k\to\infty} \min_{\lambda\vdash k}a(s_\lambda) = 0,
	$
where $\lambda\vdash k$ means that $\lambda$ is a partition of $k$.
\end{lemma}

\begin{proof}
Let $k\in\bN$. By Schur--Weyl duality, every vector space $V$ over $\bC$ satisfies
	\[
	V^{\otimes k}\cong\bigoplus\limits_{\lambda\vdash k}M_{\lambda}\otimes
	L_{\lambda}\left(  V\right)
	\]
as $S_{k}\times\operatorname*{GL}(  V)  $-modules, where
$M_{\lambda}$ is the irreducible $S_{k}$-module corresponding to the partition
$\lambda$, and $L_{\lambda}$ is the Schur functor corresponding to the
partition $\lambda$. Taking the $\operatorname*{GL}(  V)
$-character of this isomorphism, we obtain
	\[
	s_1^{k}=\sum\limits_{\lambda}\dim\left(
	M_{\lambda}\right)  \cdot s_{\lambda},
	\]
where we sum over partitions $\lambda$ satisfying $\sum_i \lambda_i=k$.
Applying the ring homomorphism $a$ to this equality, we obtain
	\begin{equation}
	\label{eq:local-2}
	a(s_1)^{k}=\sum
	\limits_{\lambda}\dim\left(  M_{\lambda}\right)  \cdot a\left(
	s_{\lambda}\right).
	\end{equation}
Writing $c_k=\min_{\lambda\vdash k}a(s_\lambda)$, we then have 
	$$
	a(s_1)^{k}  
	\geq c_k\sum_{\lambda}\dim\left(  M_{\lambda}\right)
	\geq c_k\Big(\sum_{\lambda}\dim(M_{\lambda})^2\Big)^{1/2}
	=c_k\sqrt{k!},
	$$
since the sum of the squares of the dimensions of all irreducible $S_{k}$-modules is $\vert
S_{k}\vert =k!$. It then follows that $c_k\to 0$ as $k\to\infty$.
\end{proof}

\subsection{Proof of theorem~(\ref{thm:intro-A})}
\label{subsec:proof-of-integral-E-T} 
For any Witt vector $a\in\Wsch(\bN)$, lemma (\ref{lem:goes-extinct}) implies there must be a partition
$\lambda$ such that $a(s_\lambda)=0$. In other words, the image of $a$ in $\Wsch(\Bool)$ is not $\infty$.
Therefore by (\ref{thm:intro-C})(\ref{part:disc-inv}) and the Edrei--Thoma theorem~(\ref{thm:Witt-E-T}),
the series $\srp(a)\in 1+t\bR[[t]]$ is of the form $g(t)/h(t)$, where $g(t)$ and $h(t)$ are
polynomials in $1+t\bR[t]$ such that all the complex roots of $g(t)$ are negative real numbers and all
those of $h(t)$ are positive real numbers. Now invoke the fact that if the ratio of two coprime
polynomials in $1+t\bC[t]$ has integral coefficients when expanded as a power series in $t$, then the
polynomials themselves have integral coefficients. (To show they have rational coefficients, one can
observe that $g(t)$ and $h(t)$ are invariant under any field automorphism of $\bC$, but there are also
choice-free arguments. Then to show the coefficients are integral, one can argue as in the solution to
exercise 1a of chapter 4 of~\cite{Stanley:EC1}.) \qed

\subsection{Remark}
Another way of expressing theorem~(\ref{thm:intro-A}) is that we have an isomorphism of
$\bN$-algebras
	$$
	W(\bN)\tn_\bN \bN[\eta]/(\eta^2=1) \longisomap \Wsch(\bN)
	$$
which is the natural map on the factor $W(\bN)$ and sends $1\tn\eta$ to the anti-Teichm\"uller lift 
$\anti{1}\in\Wsch(\bN)$. In fact, each factor on the left side has a compatible $\Sch$-semiring
structure making the map an isomorphism of $\Sch$-semirings.

\subsection{Remark: partition-ideal phenomena in general $W(A)$}
\label{subsec:partitional-ideal-phenomena}
As discussed in the introduction, phenomena related to $\bN^2$ and partition ideals will be present in $W(A)$ (and $\Wsch(A)$) whenever $A$ is not a ring. This is simply because, by~(\ref{prop:nonrings-map-to-Bool}) below, every non-ring $A$ admits a map to $\Bool$ and hence there is always a map $W(A)\to W(\Bool)$. This is analogous to the presence of $p$-adic phenomena in Witt vectors of rings admitting a map to a field of characteristic $p$, or equivalently rings not containing $1/p$. But in the case of $\Bool$, there is more structure: the set of maps $A\to\Bool$ has a partial order given by pointwise application of the partial order $0<1$ of $\Bool$. So in fact there is a family of partition-ideal phenomena indexed by this partially ordered set.

For example, consider the case where $A$ is a zerosumfree domain (that is, it is nonzero and
satisfies the implication $x+y=0 \Rightarrow x=y=0$ and the implication $xy=0 \Rightarrow x=0\text{
or }y=0$) and $f$ is the $\bN$-algebra map $A\to\Bool$ sending all nonzero elements to $1$. This is
the unique maximal element of the partially ordered set and hence gives a particularly natural way 
in which partition-ideal phenomena appear. The following is an immediate a consequence of~(\ref{thm:intro-B}).

\begin{corollary}
	Let $A$ be a zerosumfree domain, and let $a\in W(A)$ be a Witt vector such that $a(m_\lambda)=0$ for some $\lambda\in P$. Then there is a unique pair $(x,y)\in\bN^2$ such that $a(m_\lambda)=0$ if and only if $\lambda_{x+1} \geq y+1$. The analogous statement is true for $\Wsch(A)$ and the Schur basis $(s_\lambda)_{\lambda\in P}$.
\end{corollary}

We conclude this section with the proposition invoked above.

\begin{proposition}
\label{prop:nonrings-map-to-Bool}
\begin{enumerate}
	\item An $\bN$-module $M$ is a group if and only if $\Bool\tn_\bN M=0$.
	\item An $\bN$-algebra $A$ is a ring if (under the axiom of choice) and only if 
		there exists no $\bN$-algebra homomorphism $A\to\Bool$.
\end{enumerate}
\end{proposition}
\begin{proof}
(1): If $M$ is a group, then it is clear that $\Bool\tn_\bN M=0$. Conversely, suppose 
$\Bool\tn_\bN M=0$. Since $\Bool$ is $\bN/(2=1)$, the module
$\Bool\tn_\bN M$ is the quotient of $M$ by the equivalence relation $\approx$
generated by relations $2x=x$ for all $x\in M$. That is, $\approx$ is the 
transitive closure of the reflexive symmetric relation $\sim$ defined by
	$$
	a\sim b \Longleftrightarrow a=z+2x+y, b=z+x+2y \text{ for some }x,y,z\in M.
	$$
Now the set of additively invertible elements of $M$ is closed under taking both sums and summands;
so it follows that if $a\sim b$, then $a$ is additively invertible if and only if $b$ is additively
invertible. Since $\approx$ is the transitive closure of $\sim$, the same is true of $\approx$.
Since we have assumed $M/{\approx}$ is $0$, we have $a\approx 0$ for any element $a\in M$. Thus $a$ 
is additively invertible, and so $M$ is a group.

(2): If $A$ is a ring, then it is clear there are no maps $A\to\Bool$. Conversely, if $A$ is not
a ring, then by part (1), we have $\Bool\tn_\bN A\neq 0$, and so by
Zorn's lemma $\Bool\tn_\bN A$ has a minimal nonzero quotient $\bN$-algebra $B$. Since $B$ is 
nonzero and has no
nonzero quotients, it must be a field or $\Bool$. (See Golan~\cite{Golan:book1}, p.\ 87.) But since 
$B$ admits a map from $\Bool$, it cannot be a field.
Therefore we have $B=\Bool$, and hence there is a map $A\to\Bool$.
\end{proof}

\section{$\Wp(\Bool)$}
\label{sec:p-typical}

Let $p$ be a prime number. In this section, we determine the $p$-typical Witt vectors of $\Bool$. The result is straightforward, but we include it for completeness.

\subsection{$p$-typical Witt vectors}
\label{subsec:p-typical}
Let us recall some definitions and basic results from section 8 of~\cite{Borger:totalpos}.
Write
	$$
	\psi = \sum_i x_i^p, \quad d = \frac{\big(\sum_i x_i)^p-\sum_i x_i^p}{p},
	\quad d_{i,j}=d^{\circ i}\circ \psi^{\circ j} \in\Lambda_\bN,
	$$
for all $i,j\in\bN$, where $\circ$ denotes the plethysm operation.
The $\bN$-algebra $\Lambda_{\bN,(p)}$ of $p$-typical symmetric functions is then defined to be the
sub-$\bN$-algebra of $\Lambda_\bN$ generated by the set $\setof{d_{i,j}}{i,j\in\bN}$,
and the set of $p$-typical Witt vectors with entries in a given $\bN$-algebra $A$ is defined by
	$$
	\Wp(A) = \Alg{\bN}(\Lambda_{\bN,(p)},A).
	$$ 
By (8.3) of~\cite{Borger:totalpos}, the relations $d_{i,j}^p = pd_{i+1,j}+
d_{i,j+1}$, generate all the relations between the $d_{i,j}$.
Therefore if we write $a_{i,j}=a(d_{i,j})$, we have the identification
	\begin{equation}
	\label{eq:p-typical-W}
	\Wp(A) = \bigsetof{(a_{i,j})\in A^{\bN^2}}{a_{i,j}^p = pa_{i+1,j}+ a_{i,j+1}}.
	\end{equation}
As shown in~\cite{Borger:totalpos}, there is a 
unique structure of a composition $\bN$-algebra on $\Lambda_{\bN,(p)}$ compatible with that of 
$\Lambda_\bN$. As always, $\Wp(A)$ is then naturally an $\bN$-algebra with an action of 
$\Lambda_{\bN,(p)}$ defined by $f(a)\:g\mapsto a(g\circ f)$ for all $f,g\in\Lambda_{\bN,(p)}$
and $a\in\Wp(A)$. Therefore, in terms of (\ref{eq:p-typical-W}), we have
	\begin{equation}
		\label{eq:p-typical-lambda-action}
		d(a)_{i,j} = a_{i+1,j}, \quad \psi(a)_{i,j}=a_{i,j+1}.
	\end{equation}

\subsection{$\Wp(\Bool)$}
In the particular case $A=\Bool$, equation~(\ref{eq:p-typical-W}) simplifies to
	\begin{equation}
	\label{eq:p-typical-W(Bool)}
	\Wp(\Bool) 
		= \bigsetof{(a_{i,j})\in \Bool^{\bN^2}}{a_{i,j}=0 \Leftrightarrow a_{i,j+1}=a_{i+1,j}=0}.
	\end{equation}
Since any $a_{i,j}$ is either $0$ or $1$, we see that $\Wp(\Bool)$ is identified with the set
of subsets $I\subseteq\bN^2$ satisfying $(i,j)\in I \Longleftrightarrow (i+1,j),(i,j+1)\in I$.
Since any such subset $I$, if nonempty, is determined by its unique element $(i,j)$ with $i+j$ 
minimal, we have a bijection
	\begin{equation}
	\label{map:p-typical-coords}
	\bN^2\cup\{\infty\}\longisomap\Wp(\Bool)
	\end{equation}
sending $(x,y)$ to the Witt vector $(a_{i,j})$ with $a_{i,j}=0 \Leftrightarrow i\geq x, j\geq y$ 
and sending $\infty$ to the Witt vector with $a_{i,j}=1$ for all $i,j$. 
By~(\ref{eq:order-determined-by-generators}), this bijection is 
an isomorphism of partially ordered sets, where $\Wp(\Bool)$ has the partial order $\wleq$,
where $\bN^2$ has the componentwise order
	$$
	(x,y)\leq (x',y') \Longleftrightarrow x\leq x' \text{ and } y\leq y',
	$$
and where $\infty$ is a terminal element.
Under this bijection, 
(\ref{eq:p-typical-lambda-action}) becomes
	\begin{equation}
	\label{eq:d-psi-translations}
	d(x,y)=(x-1,y), \quad \psi(x,y)=(x,y-1), \quad d(\infty)=\infty, \quad \psi(\infty)=\infty,
	\end{equation}
where we understand that negative coordinates are rounded up to zero.

\begin{proposition}
\label{prop:p-typical}
The $\bN$-algebra structure on $\bN^2\cup\{\infty\}$ inherited 
		from (\ref{map:p-typical-coords})
		satisfies the following properties for all $x,x',y,y'\in\bN$:
		\begin{enumerate}
			\item\label{(e)} The image of $n\in\bN$ under the unique algebra map 
				$\bN\to\bN^2\cup\{\infty\}$ is $(n,0)$ if $n\leq 1$ (or rather $n\leq 2$ if $p=2$)
				and is $\infty$ otherwise.
			\item\label{(f)} $(x,0)+(x',0)=\infty$ if $x,x'\geq 2$ 
			\item\label{(g)} $(0,y)+(0,y')=(0,\max\{y,y'\})$
			\item\label{(h)} $(x,0)+(0,y)=(x,y)$
			\item\label{(i)} $\infty+z=\infty$ for all $z\in\bN^2\cup\{\infty\}$
			\item\label{(j)} $(x,0)(x',0)= \infty$ for $x,x'\geq 2$
			\item\label{(k)} $(0,y)(0,y') =(0,\min\{y,y'\})$
			\item\label{(l)} $(x,0)(0,y)=(0,y)$ if $x\geq 1$
			\item\label{(m)} $\infty(0,y)=(0,y)$
			\item\label{(n)} $\infty(x,0)=\infty$ if $x\geq 1$.
		\end{enumerate}
\end{proposition}
\begin{proof}
We will identify any element of $\bN^2\cup\{\infty\}$ with its image in $\Wp(\Bool)$ 
under the bijection~(\ref{map:p-typical-coords}).

(\ref{(e)}): Let $\varphi$ denote the map $\bN\to\Wp(\bN)$. 
Since it is $\Lambda_{\bN,(p)}$-equivariant, we have
	$$
	(\varphi(n))_{i,0}=\big(d^{\circ i}(\varphi(n))\big)_{0,0} = 
	\big(\varphi(d^{\circ i}(n))\big)_{0,0} = d^{\circ i}(n).
	$$
Now for $n\in\bN$, we have $d(n)=(n^p-n)/p$. Thus for $p\geq 3$, we have $d(0)=d(1)=0$ and $d(m)\geq 2$ for
$m\geq 2$. Therefore if $n\leq 1$, the smallest $i$ such that $d^{\circ i}(n)=0$ is $n$, and no such $i$
exists if $n\geq 2$. This finishes the proof for $p\geq 3$. The statement for $p=2$ follows similarly from
$d(0)=d(1)=0$, $d(2)=1$, and $d(m)\geq 3$ for $m\geq 3$.

(\ref{(f)}): By part (\ref{(e)}), we have $(x,0)+(x',0) \wgeq (2,0)+(2,0)=2+2 = \infty$.

(\ref{(g)}): Since the power sum $d_{0,j}=\psi^{\circ j}$ is primitive for the co-addition map
$\Delta^+$, we have $((0,y)+(0,y'))_{0,j} = (0,y)_{0,j}+ (0,y')_{0,j}$. This evaluates to $0$ if
and only if $y,y'\leq j$. Therefore $((0,y)+(0,y'))_{i,j}$ is zero for
$(i,j)=(0,\max\{y,y'\})$ but not for any smaller $(i,j)$. Thus we have
$(0,y)+(0,y')=(0,\max\{y,y'\})$.

(\ref{(h)}): If either $x$ or $y$ is zero, the statement follows because $(0,0)$ is the additive
identity, by part (\ref{(e)}). So assume $x,y\geq 1$.
By~(\ref{eq:d-psi-translations}),
it is enough to show $\psi^{\circ y-1}((x,0)+(0,y))=(x,1)$. Because we have
	$$
	\psi^{\circ y-1}((x,0)+(0,y))=(x,0)+(0,1),
	$$
it is enough to show $(x,0)+(0,1)=(x,1)$ for $x\geq 1$.
Since we have 
	$$\psi((x,0)+(0,1))=\psi(x,0)+\psi(0,1)=(x,0)+(0,0)=(x,0),$$
$(x,0)+(0,1)$ is either $(x,0)$ or $(x,1)$, again by~(\ref{eq:d-psi-translations}).
To rule out $(x,0)$, we can show
	$$d((x,0)+(0,1)) \wgeq (x-1,1).$$
This holds because we have
	$$
	\Delta^+(d) = d\tn 1 + 1\tn d +\cdots 
	$$
and hence 
	\begin{align*}
	d((x,0)+(0,1)) &= d(x,0)+d(0,1)+\cdots \\ &\wgeq d(x,0)+d(0,1) \\ &=(x-1,0)+(0,1) \\ &=(x-1,1),
	\end{align*}
by~(\ref{pro:inequality-comparison}) and induction on $x$.

(\ref{(i)}): By~(\ref{pro:inequality-comparison}),
we have $\infty+z\wgeq \infty$ and hence $\infty+z=\infty$.

(\ref{(j)}): Since $\wgeq$ is an $\Alg{\bN}$-preorder,
we have $(x,0)(x',0)\wgeq (2,0)(2,0)$. Therefore it is enough to show $(2,0)^2=\infty$.
We have
	\begin{equation}
		\label{eq:d-sum-law}
		\Delta^\times(d) = \psi\tn d + d\tn \psi + p d\tn d
	\end{equation}
and hence
	\begin{align*}
	d((2,0)(2,0)) &= \psi(2,0)d(2,0) + d(2,0)\psi(2,0) + pd(2,0)^2 \\
		&= (2,0)(1,0) + (1,0)(2,0) +p(1,0)^2 \\
		&= (2,0)+(2,0) + p(1,0) \\
		&= \infty + p(1,0) \\
		&= \infty,
	\end{align*}
by~(\ref{eq:d-psi-translations}) and parts~(\ref{(e)}),~(\ref{(f)}), and~(\ref{(i)}).
It follows that $(0,2)^2=\infty$.

(\ref{(k)}): For any $j$, the element $d_{0,j}=\psi^{\circ j}$ is group-like under $\Delta^\times$.
Therefore we have 
	$$
	\big((0,y)(0,y')\big)_{0,j} = \big((0,y)\big)_{0,j}\cdot \big((0,y')\big)_{0,j},
	$$
which is $0$ if and only if $j\geq\min\{y,y'\}$.
Therefore $\big((0,y)(0,y')\big)_{i,j}$ is zero for $(i,j)=(0,\min\{y,y'\})$ but not for any
smaller $(i,j)$. It follows that $(0,y)(0,y')$ is $(0,\min\{y,y'\})$.

(\ref{(l)}): As in part~(\ref{(k)}), we have $\big((x,0)(0,y)\big)_{0,j}=(x,0)_{0,j}(0,y)_{0,j}$, for any
$j\in\bN$. Since $x\geq 1$, we have $(x,0)_{0,j}=1$, and so $\big((x,0)(0,y)\big)_{0,j}$ is $(0,y)_{0,j}$.
Thus we have $(x,0)(0,y)=(0,y)$.

(\ref{(m)}): As in part (\ref{(l)}), we have $\big(\infty(0,y)\big)_{0,j}=(0,y)_{0,j}$, and so
$\infty(0,y)=\infty$.

(\ref{(n)}): By part~(\ref{(e)}), the element $(1,0)$ is the multiplicative identity; so the result is
clear for $x=1$. If $x\geq 2$, then~(\ref{eq:d-sum-law}) and~(\ref{pro:inequality-comparison}) imply
$$
d(\infty(x,0))\wgeq \psi(\infty)d(x,0)=\infty(x-1,0),
$$ 
which evaluates to $\infty$ by induction.
Therefore we have $d(\infty(x,0))=\infty$ and hence $\infty(x,0)=\infty$. \end{proof}

\begin{proposition}
Consider the canonical map $W(\Bool)\to\Wp(\Bool)$, in the coordinates given by the bijections 
(\ref{map:W(Bool)-coords}) and (\ref{map:p-typical-coords}).
If $x\leq 1$ (or rather $x\leq 2$ for $p=2$), then an element $(x,y)\in\bN^2$ is sent to 
$(x,z)$, where $z$ is the smallest element of $\bN$ such that $p^{z}> y$. All other elements are sent
to $\infty$.
\end{proposition}
\begin{proof}
It is clear that $\infty$ is sent to $\infty$. For elements of the form $(x,0)$, the result
follows from part (\ref{(e)}) of (\ref{prop:p-typical}). Now consider an element of the form 
$(0,y)\in W(\Bool)$. Then
$(0,y)(\psi_{p^j})=0$ if and only if $p^j>y$. Therefore the image $a\in\Wp(\Bool)$ of $(0,y)$
is $(0,z)$, where $z$ is the smallest element of $\bN$ such that $p^z>y$.

For an element of the form $(x,y)\in W(\Bool)$, we have $(x,y)=(x,0)+(0,y)$. So the result
follows from the previous cases and parts (\ref{(h)}) and (\ref{(i)}) of (\ref{prop:p-typical}). 
\end{proof}

\subsection{Remarks}
The set of $p$-typical symmetric functions of length $k$ is defined to be the sub-$\bN$-algebra
of $\Lambda_\bN$ generated by the subset $\setof{d_{i,j}}{i+j\leq k}$, and the functor it 
represents is defined to be $W_{(p),k}$, the $p$-typical Witt functor of length $k$.
Then we have the following analogue of (\ref{eq:p-typical-W(Bool)}):
	\begin{equation*}
	W_{(p),k}(\Bool) 
		= \bigsetof{(a_{i,j})\in \Bool^{T_k}}{a_{i,j}=0 \Leftrightarrow a_{i,j+1}=a_{i+1,j}=0
		\text{ for }i+j\leq k-1},
	\end{equation*}
where $T_k=\setof{(i,j)\in\bN^2}{i+j\leq k}$. But there is no analogue of 
(\ref{map:p-typical-coords}). Indeed, the triangular array
	$$
	(a_{i,j}) = \left(\begin{array}{ccc} 0 \\ 1 & 1 \\ 1 & 1 & 0 \end{array}\right)
	$$
has two blocks of zeros yet defines an element of $W_{(p),2}(\Bool)$.

This also shows that the map $W_{(p),4}(\Bool)\to W_{(p),3}(\Bool)$ is not surjective, because
the element above does not lift to $W_{(p),4}(\Bool)$.
In fact, as one easily checks, this element lifts to $W_{(p),3}(\bRpl)$, and so
the map $W_{(p),4}(\bRpl)\to W_{(p),3}(\bRpl)$ is not surjective either.

\section{Complements: countability of some Witt vector algebras}
\label{sec:countability}

If $A$ is uncountable, then both $W(A)$ and $\Wsch(A)$ are uncountable because the Teichm\"uller map
$A\to W(A)$ is injective. If $A$ is countable, we have the following partial result when it admits
an embedding into $\bRpl$:

\begin{theorem}
	\label{thm:density}
	Let $A$ be a countable sub-$\bN$-algebra of $\bRpl$. 
		\begin{enumerate}
			\item If $A$ is not dense, then $W(A)$ and $\Wsch(A)$ are countable. 
			\item If $A$ is dense and of the form $B\cap\bRpl$
				for some subring $B\subseteq\bR$, then $W(A)$ and $\Wsch(A)$ are uncountable.
		\end{enumerate}
\end{theorem}

\begin{proof}
(1): As in (\ref{subsec:proof-of-integral-E-T}), the series associated to any element of $\Wsch(A)$
is of the form $f(t)/g(t)$, where $f(t)$ and $g(t)$ are polynomials with coefficients
in the smallest subfield $K\subseteq \bR$
containing $A$. Since $K$ is countable, there are only countably many such series.
Further, $W(A)$ is countable because it is a subset of $\Wsch(A)$.

(2): For the same reason, here it is enough to show $W(A)$ is uncountable.
First observe that a Witt
vector $y\in W(\bR)$ lies in $W(\bRpl)$ if and only if, for all $k\geq 0$, there exists a Witt vector $x\in
W(\bRpl)$ such that $\srp(y)\equiv \srp(x) \bmod t^{k+1}$. Indeed, to have $y\in W(\bRpl)$ it is
enough to have $y(m_\lambda)\geq 0$ for all partitions $\lambda$. But each such $m_\lambda$ is a polynomial
in finitely many elementary symmetric functions $e_1,\dots,e_k$. Choosing
$x\in W(\bRpl)$ such that $\srp(y)\equiv \srp(x) \bmod t^{k+1}$, we have
$y(m_\lambda)=x(m_\lambda)\geq 0$.
 
Combining this with the equality $W(B\cap\bRpl)=W(B)\cap W(\bRpl)$, we see that it is enough to show there are
uncountably many series in $1+tB[[t]]$ that can be approximated arbitrarily well in the $t$-adic topology by
series of the form $\srp(x)$, where $x\in W(\bRpl)$. In fact, it will be convenient to consider
approximations by a slightly restricted class of series. So first let $W'(\bR)$ denote the subset of
$1+t\bR[[t]]$ consisting of series of the form $\prod_j(1+b_jt)$, for some reals $b_1>b_2>\cdots>0$. Then 
by~(\ref{cor:totalpos-summary}), we have
$W'(\bR)\subseteq\srp(W(\bRpl))$. Second, for any integer $k\geq 0$, let $W'_k(B)$ denote the intersection 
of $1+tB[[t]]/(t^{k+1})$ and the image of the truncation map 
	$$
	W'(\bR)\to 1+t\bR[[t]]/(t^{k+1}), \quad f(t)\mapsto f(t)\bmod t^{k+1}.
	$$
Then the inverse
limit $\lim_k W'_k(B)$, viewed as a subset of $1+t\bR[[t]]$, consists of series with coefficients in $B$ that
can be approximated arbitrarily well in the $t$-adic topology by elements of $W'(\bR)$. By the remarks above, 
it is contained in $\srp(W(B\cap\bRpl))$.

Therefore it is enough to show that $\lim_k W'_k(B)$ is uncountable. Since $W'_0(B)$ is nonempty, it is enough
to show that the fibers of the truncation maps $W'_k(B)\to W'_{k-1}(B)$ have infinitely many elements (or even
two). So consider an element $\bar{f}(t)\in W'_{k-1}(B)$, and choose a lift $f(t)=\prod_j(1+b_jt)\in W'(\bR)$.
Then for all sufficiently small $\varepsilon\in\bR$, the polynomial $\varepsilon t^k+\prod_{j=1}^{k}(1+b_jt)$ is still of the form $\prod_{j=1}^{k}(1+b'_jt)$,
where $b'_1>\cdots >b'_k>b_{k+1}$. For such $\varepsilon$, write
	$$
	f_\varepsilon(t)=\prod_{j=1}^{k}(1+b'_jt)\prod_{j=k+1}^{\infty}(1+b_jt) \in W'(\bR).
	$$
Then we have 
	\begin{equation}
		\label{eq:loc-approx}
		f_\varepsilon(t)=f(t)+\varepsilon t^k \bmod t^{k+1}.
	\end{equation}
Therefore the series $f_\varepsilon(t)$ are distinct for distinct $\varepsilon$,
and they reduce to $\bar{f}(t)$ modulo $t^{k}$. Further, we
have $f_\varepsilon(t)\in W'_k(B)$ if (and only if) the $k$-th coefficient of $f_\varepsilon(t)$ lies in $B$.
Since $A$ is dense in $\bRpl$ and since $B$ contains $\pm A$, we know that
$B$ is dense in $\bR$. Therefore equation (\ref{eq:loc-approx}) guarantees that there are infinitely many 
sufficiently small $\varepsilon\in\bR$ such that the $k$-th coefficient of $f_\varepsilon(t)$ lies in $B$. 
In other words, the fiber of the map $W'_k(B)\to W'_{k-1}(B)$ over the arbitrary element $\bar{f}(t)$ is 
infinite. 
\end{proof}

\subsection{Examples}
Since one can embed $\bN[x_1,x_2,\dots]$ as a non-dense sub-$\bN$-algebra of $\bRpl$, 
we see that $W(A)$ and $\Wsch(A)$ are countable for any
sub-$\bN$-algebra $A$ of $\bN[x_1,x_2,\dots]$. In particular, we have countability for
$A=\bN[\alpha]$, where $\alpha$ is any transcendental number.
This holds even when $|\alpha|<1$, in which case $A$ is
dense in $\bRpl$. So the condition $A=B\cap\bRpl$ in (\ref{thm:density})(2) cannot be removed. Similarly,
$W(A)$ and $\Wsch(A)$ are countable for the dense subalgebra $A=\bN[2-\sqrt{2}]\subseteq\bRpl$, because
$2-\sqrt{2}$ is conjugate to $2+\sqrt{2}$, which is greater than $1$.

Examples where $W(A)$ and $\Wsch(A)$ are uncountable include $A=\bN[1/m]$ for any $m=2,3,\dots$
and $A=\bN[\sqrt{2}-1]=\bZ[\sqrt{2}]\cap\bR_+$.

Finally, we emphasize that there are dense sub-$\bN$-algebras $\bRpl$ to which 
(\ref{thm:density})(2) does not 
apply---for example $\bN[\sqrt{3}-1]$. At the time of this writing, we do not know whether either
$W(A)$ or $\Wsch(A)$ is countable for any such algebras $A$.

\begin{theorem}
\label{thm:Wn-count-uncount}
	Let $n$ be a nonnegative integer, and write $A=\bN/(n=n+1)$.
	\begin{enumerate}
		\item $W(A)$ is countable if and only if $n\leq 2$.
		\item $\Wsch(A)$ is countable if and only if $n\leq 2$.
	\end{enumerate}
\end{theorem}

Note the contrast between the uncountability for $n\geq 3$ and the countability of $W\left(\bN\right)$
and $\Wsch\left(\bN\right)$, given by~(\ref{cor:totalpos-summary}) and~(\ref{thm:intro-A}).

\begin{proof}
(1): When $n=0$, the result is clear, and
when $n=1$, it follows immediately from (\ref{subsec:explicit-W(Bool)}).

Now consider the case $n=2$. For any $a\in W(A)$, write 
	$$
	\pker(a)=\setof{\lambda\in P}{a(m_\lambda)=0}.
	$$
A basic result we will use more than once is the following. Let $\lambda$ be a
partition, and write $d_\lambda(a)=\sum_{\nu\not\in\pker(a)} b^\nu_{\lambda\lambda}$, where
$m_\lambda^2 = \sum_{\nu\in P} b^\nu_{\lambda\lambda} m_\nu$; then we have
	\begin{equation}
	\label{eq:local-238}
	d_\lambda(a) \geq 2 \Longrightarrow a(m_\lambda)=2
	\end{equation}
Indeed, since all elements in $A$ are multiplicatively idempotent, we have
	$$
	a(m_\lambda) = a(m_\lambda)^2 = a(m_\lambda^2) = \sum_\nu b^\nu_{\lambda\lambda}a(m_\nu)
		\geq d_\lambda(a) \geq 2
	$$
and hence $a(m_\lambda)=2$.

Now consider the map $\varphi\:W(A)\to W(\Bool)$ induced by the unique $\bN$-algebra map $A\to\Bool$. Since
$W(\Bool)$ is countable, it is enough to show that for each $z\in W(\Bool)$, there are only finitely many
$a\in W(A)$ such that $\varphi(a)=z$.

First consider the case $z=\infty$. Then for any nonzero partition $\lambda$, the symmetric function
$m_\lambda^2$ has at least two monomial constituents, by~(\ref{pro:W(Bool)-lr-extremals}). Further, neither
constituent is in $\pker(a)$ since $\pker(a)=\emptyset$. Therefore we have $a(m_\lambda)=2$, by
(\ref{eq:local-238}). Since we also have $a(m_0)=a(1)=1$, there is at most one Witt vector $a$ such that
$\varphi(a)=\infty$.

It remains to consider the case $z=(x,y)$, where $x,y\in\bN$. It is enough to show that
any $a\in\varphi^{-1}(z)$ is determined by its values $a(m_{u^x})$ for $u\leq 2y+1$,
because there are only finitely many possibilities for such values. We will do this by saying 
explicitly how
these values determine each $a(m_\lambda)$; and we will do this by 
considering a nested sequence of cases in $\lambda$.

If $\lambda_{x+1}\geq y+1$, then we have $\lambda\in\pker(a)$ and hence $a(m_\lambda)$ is
zero. In particular, the value of $a(m_\lambda)$ is determined. Therefore we may assume
	$
	\lambda_{x+1}\leq y.
	$
First suppose $\lambda_{x+1}>0$. 
Clearly $m_\lambda^2$ has the monomial constituent
	$$
	\nu=(2\lambda_1, \dots, 2\lambda_x, \lambda_{x+1}, \lambda_{x+1}, \dots).
	$$
Since $\lambda_{x+1}>0$, if we view $\nu$ as a vector, we can write it as the sum of two distinct 
permutations of 
$\lambda$. In other words, the $m_\nu$ term in $m_\lambda^2$ is a cross term, and so its
coefficient $b^\nu_{\lambda\lambda}$ is at least $2$.
Further we have $\nu_{x+1}=\lambda_{x+1}\leq y$, and so $\nu\not\in\pker(a)$.
Therefore we have $d_\nu(a)\geq 2$, and hence $a(m_\lambda)$ can only be $2$, 
by~(\ref{eq:local-238}). In particular, its value is determined when $\lambda_{x+1}>0$.

Therefore we may assume
	$
	\lambda_{x+1}=0.
	$
Now suppose there exists $i$ with $1\leq i\leq x-1$ such that $\lambda_i > \lambda_{i+1}$.
Then $m_\lambda^2$ has two monomial constituents $\nu$:
	\begin{equation*}
	(2\lambda_1, \dots, 2\lambda_x),\quad
	(2\lambda_1, \dots, 2\lambda_{i-1}, \lambda_i + \lambda_{i+1}, \lambda_{i+1} + 
		\lambda_i, 2\lambda_{i+2}, \dots, 2\lambda_x).
	\end{equation*}
Since $\lambda_i > \lambda_{i+1}$, the two differ at position $i$ and are hence distinct.
Also, both have length at most $x$. So neither lies in $I_{(x,y)}=\pker(a)$.
Thus we have $d_\lambda(a)\geq 2$, and again by~(\ref{eq:local-238}) the value of $a(m_\lambda)$ is 
determined.

Therefore we may assume there is no such $i$, and hence that $\lambda=u^x$, for some $u$. 
For any integers $p,q \geq
y+1$, the only constituent of $m_{p^x}m_{q^x}$ which does not belong to $I_{(x, y)}$ is $(p+q)^x$,
and this constituent appears with coefficient $1$. This implies $f(m_{p^x}) f(m_{q^x}) =
f(m_{(p+q)^x})$, and so by induction the values $f(m_{u^x})$ for $u=y+1,\dots, 2y+1$
determine $f(m_{u^x})$ for any $u\geq 2y+2$. This finishes the proof when $n=2$.

Finally, consider the case $n\geq 3$.

Let $U$ be any subset of the set of all nonzero partitions. Clearly, there are uncountably many choices for
this $U$. To show $W(A)$ is uncountable, we will construct an element of $W(A)$ corresponding to $U$ and
show that all such elements of $W(A)$ are distinct.

Define an $\bN$-linear map $f:\Lambda_\bN \to A$ by
	$$
	f\left(m_\lambda\right) =
	\begin{cases}
	1 &\text{if }\lambda = 0 \\ n-1&\text{if }\lambda\in U \\ n&\text{otherwise,}
	\end{cases}
	$$
for every partition $\lambda$.
Clearly, the subset $U$ can be uniquely reconstructed from this map $f$, since $n-1\neq n$ in $A$. We will now show that $f$ is an element of $W(A)$. To prove this, we must verify that $f$ is an $\bN$-algebra homomorphism. Since $f\left(1\right)=f\left(m_{0}\right)=1$, it is enough to prove that 
$f\left(m_\lambda m_\mu\right) = f\left(m_\lambda\right) f\left(m_\mu\right)$ for any partitions $\lambda$ and $\mu$.

If $\lambda=0$ or $\mu=0$, this is obvious, since $m_{0}=1$. 
So we may assume $\lambda$ and $\mu$ are nonzero.
Therefore we have
$$f\left(m_\lambda\right) f\left(m_\mu\right) \geq \left(n-1\right)\left(n-1\right) \geq n$$
since $n\geq 3$. Because every element of $A$ which is at least $n$ must be equal to $n$, this yields $f\left(m_\lambda\right) f\left(m_\mu\right) = n$.

On the other hand, (\ref{pro:W(Bool)-lr-extremals})(3) shows that $m_\lambda m_\mu$ has at least two distinct nonzero constituents. That is, there exist two distinct nonzero partitions $\alpha$ and $\beta$ such that $m_\lambda m_\mu \geq m_\alpha + m_\beta$. 
Therefore we have
	$$
	f\left(m_\lambda m_\mu\right) \geq f\left(m_\alpha + m_\beta\right) 
		= f(m_\alpha) + f(m_\beta)
        \geq \left(n-1\right)+\left(n-1\right) 
		\geq n
	$$
since $n\geq 3$ again. As above, this yields $f\left(m_\lambda m_\mu\right) = n$
and hence
	$$
	f\left(m_\lambda m_\mu\right) =n= f\left(m_\lambda\right) f\left(m_\mu\right).
	$$
Thus $f$ is an element of $W(A)$.

Therefore $W(A)$ has at least as many elements as there are
subsets of nonzero partitions. Hence, it has uncountably many.

(2): For $n\neq 2$, the arguments above also work in the Schur basis, instead of the monomial one.
One simply appeals to (\ref{subsec:explicit-Wsch(Bool)}) instead of
(\ref{subsec:explicit-W(Bool)}), and (\ref{pro:Wsch(Bool)-lr-extremals})(3) instead of
(\ref{pro:W(Bool)-lr-extremals})(3). 
When $n=2$, the modifications we need to make are more elaborate. So we shall not record
them here.
\end{proof}

\frenchspacing
\bibliography{references}
\bibliographystyle{plain}

\end{document}